\documentclass[11pt]{amsart}
\usepackage{amssymb,amsthm,amsmath,amstext}
\usepackage{mathrsfs}  
\usepackage{bm}        
\usepackage{mathtools} 

\usepackage[left=1.28in,top=.8in,right=1.28in,bottom=.7in]{geometry}

\usepackage{graphicx}

\usepackage[all]{xy}

\theoremstyle{plain}
\newtheorem{theorem}{Theorem}[section]
\newtheorem{prop}[theorem]{Proposition}
\newtheorem{lemma}[theorem]{Lemma}
\newtheorem{cor}[theorem]{Corollary}

\newtheorem{theoremintro}{Theorem}

\theoremstyle{definition}
\newtheorem{defin}[theorem]{Definition}
\newtheorem{example}[theorem]{Example}

\theoremstyle{remark}
\newtheorem{remark}[theorem]{Remark}

\newcommand{\sheaf}[1]{\mathscr{#1}}
\newcommand{\LL}{\sheaf{L}}
\newcommand{\OO}{\sheaf{O}}
\newcommand{\MM}{\sheaf{M}}
\newcommand{\EE}{\sheaf{E}}
\newcommand{\FF}{\sheaf{F}}

\newcommand{\NN}{\sheaf{N}}
\newcommand{\VV}{\sheaf{V}}

\renewcommand{\AA}{\sheaf{A}}
\newcommand{\BB}{\sheaf{B}}
\newcommand{\PP}{\sheaf{P}}

\newcommand{\DD}{\sheaf{D}}

\newcommand{\Norm}{N}
\newcommand{\residue}{\partial}

\newcommand{\divisor}{\mathrm{div}}
\newcommand{\Brtwo}{{}_2\mathrm{Br}}

\newcommand{\mm}{\mathfrak{m}}

\newcommand{\QF}{\mathrm{QF}}
\newcommand{\rad}{\mathrm{rad}}
\newcommand{\CliffB}{\sheaf{C}}

\newcommand{\qf}[1]{<\!#1\!>}
\newcommand{\Az}{\mathrm{Az}}
\newcommand{\Quadrics}{\mathrm{Quad}}

\DeclareMathOperator{\HHom}{\sheaf{H}\!\mathit{om}}
\DeclareMathOperator{\Hom}{\mathrm{Hom}}
\DeclareMathOperator{\EEnd}{\sheaf{E}\!\mathit{nd}}

\newcommand{\Group}[1]{\mathbf{#1}}
\newcommand{\GL}{\Group{GL}}
\newcommand{\SL}{\Group{SL}}
\newcommand{\PGL}{\Group{PGL}}
\newcommand{\Orth}{\Group{O}}
\newcommand{\SOrth}{\Group{SO}}
\newcommand{\GOrth}{\Group{GO}}
\newcommand{\GSOrth}{\Group{GSO}}

\newcommand{\PGOrth}{\Group{PGO}}
\newcommand{\PGSOrth}{\Group{PGSO}}

\newcommand{\PSOrth}{\Group{PSO}}

\newcommand{\CliffAlg}{\mathscr{C}}
\newcommand{\CliffZ}{\mathscr{Z}}

\newcommand{\isom}{\cong}

\newcommand{\Z}{\mathbb Z}
\newcommand{\A}{\mathbb A}
\newcommand{\C}{\mathbb C}
\renewcommand{\P}{\mathbb P}
\newcommand{\Q}{\mathbb Q}

\DeclareMathOperator{\Aut}{\mathrm{Aut}}
\DeclareMathOperator{\Br}{\mathrm{Br}}

\DeclareMathOperator{\Pic}{\mathrm{Pic}}
\DeclareMathOperator{\rk}{\mathrm{rk}}

\DeclareMathOperator{\Spec}{\mathrm{Spec}}

\DeclareMathOperator{\Ext}{\mathrm{Ext}}

\newcommand{\muu}{\bm{\mu}}
\newcommand{\inv}{^{-1}}
\newcommand{\mult}{^{\times}}
\newcommand{\dual}{^{\vee}}
\newcommand{\exterior}{{\textstyle \bigwedge}}
\newcommand{\tensor}{\otimes}

\newcommand{\vp}{\varphi}

\newcommand{\sig}{\sigma}
\newcommand{\bslash}{\smallsetminus}
\newcommand{\pullback}{^{*}}
\newcommand{\pushforward}{_{*}}
\newcommand{\transpose}{^{{t}}}

\newcommand{\id}{\mathrm{id}}
\newcommand{\mapto}[1]{\xrightarrow{#1}}
\newcommand{\ol}[1]{\overline{#1}}

\newcommand{\wt}[1]{\widetilde{#1}}
\newcommand{\subsetto}{\hookrightarrow}
\newcommand{\et}{\mathrm{\acute{e}t}}

\newcommand{\Dynkin}[1]{\mathsf{#1}}

\newcommand{\linedef}[1]{\textsl{#1}}

\newcommand{\Lie}[1]{\mathfrak{#1}}
\newcommand{\ep}{\epsilon}

\renewcommand{\SOrth}{\Orth^+}
\renewcommand{\GSOrth}{\GOrth^+}
\renewcommand{\PGSOrth}{\PGOrth^+}

\newcommand{\EOrth}{\Group{E}\Orth}
\newcommand{\Gm}{\Group{G}_{\text{m}}}
\newcommand{\Ga}{\Group{G}_{\text{a}}}

\newcommand{\obar}[1]{\mkern 1.5mu\overline{\mkern-1.5mu#1\mkern-1.5mu}\mkern 1.5mu}
\newcommand{\wh}[1]{\widehat{#1}}

\newcommand{\PProj}{\mathbf{Proj}\,}
\newcommand{\graded}{^{\bullet}}
\newcommand{\fl}{\mathrm{fppf}}
\newcommand{\similar}{\simeq}

\newcommand{\Het}{H_{\et}}
\newcommand{\Div}{\mathrm{Div}}
\newcommand{\Cl}{\mathrm{Cl}}

\usepackage{hyperref}
\hypersetup{pdftitle={Quadric surface bundles over surfaces}}
\hypersetup{pdfauthor={Asher Auel}}
\hypersetup{colorlinks=true,linkcolor=blue,anchorcolor=blue,citecolor=blue,letterpaper=true}

\begin{document}

\title[Quadric surface bundles over surfaces]{Quadric surface bundles over surfaces}

\author[Auel]{Asher Auel}
\address{Department of Mathematics \& Computer Science \\ %
Emory University \\ %
400 Dowman Drive~NE \\ %
Atlanta, GA 30322, USA}
\email{auel@mathcs.emory.edu, parimala@mathcs.emory.edu, suresh@mathcs.emory.edu}

\author[Parimala]{R.\ Parimala}

\author[Suresh]{V.\ Suresh}

\date{\today}


\begin{abstract}
Let $f : T \to S$ be a finite flat morphism of degree 2 of regular
integral schemes of dimension $\leq 2$ (with 2 invertible), having
regular branch divisor $D \subset S$.  We establish a bijection
between Azumaya quaternion algebras on $T$ and quadric surface bundles
with simple degeneration along $D$.  This is a manifestation of the
exceptional isomorphism ${}^2\Dynkin{A}_1=\Dynkin{D}_2$ degenerating
to the exceptional isomorphism $\Dynkin{A}_1=\Dynkin{B}_1$.  In one
direction, the even Clifford algebra yields the map.  In the other
direction, we show that the classical algebra norm functor can be
uniquely extended over the discriminant divisor.  Along the way, we
study the orthogonal group schemes, which are smooth yet nonreductive,
of quadratic forms with simple degeneration.  Finally, we provide two
applications: constructing counter-examples to the local-global
principle for isotropy, with respect to discrete valuations, of
quadratic forms over surfaces; and a new proof of the global Torelli
theorem for very general cubic fourfolds containing a plane.
\end{abstract}

\maketitle

\section*{Introduction}

A quadric surface bundle $\pi : Q \to S$ over a scheme $S$ is the flat
fibration in quadrics associated to a line bundle-valued quadratic
form $q : \EE \to \LL$ of rank 4 over $S$.  A natural class of quadric
surface bundles over $\P^2$ appearing in algebraic geometry arise from
cubic fourfolds $Y \subset \P^5$ containing a plane.  Projection from
the plane $\pi : \tilde{Y} \to \P^2$ defines a quadric surface bundle,
where $\tilde{Y}$ is the blow-up of $Y$ along the plane.  Such quadric
bundles have degeneration along a sextic curve $D \subset \P^2$.  If
$Y$ is sufficiently general then $D$ is smooth and the double cover $T
\to \P^2$ branched along $D$ is a K3 surface of degree 2.  
Over the surface $T$, the even Clifford algebra $\CliffB_0$
associated to $\pi$ becomes an Azumaya quaternion algebra representing
a Brauer class $\beta \in \Brtwo(T)$.
For $T$ sufficiently general, the association $Y \mapsto (T,\beta)$ is
injective:\ smooth cubic fourfolds $Y$ and $Y'$ giving rise to
isomorphic data $(T,\beta)\isom (T',\beta')$ are linearly isomorphic.
This result was originally obtained via Hodge theory by
Voisin~\cite{voisin:cubic_fourfolds} in her celebrated proof of the
global Torelli theorem for cubic fourfolds.

In this work, we provide an algebraic generalization of this result to
any regular integral scheme $T$ of dimension $\leq 2$ with 2
invertible, which is a finite flat double cover of a regular scheme
$S$ with regular branch divisor $D \subset S$.  Let $\Quadrics_2(T/S)$
denote the set of $S$-isomorphism classes of quadric surface bundles
over $S$ with simple degeneration 
along $D$ and discriminant cover $T\to S$ (see \S\ref{sec:simple} for
details).  Let $\Az_2(T/S)$ denote the set of $\OO_T$-isomorphism
classes of Azumaya algebras of degree $2$ over $T$ with generically
trivial corestriction to $S$ (see \S\ref{sec:cores} for details).  The
even Clifford algebra functor yields a map $\CliffAlg_0 :
\Quadrics_2(T/S) \to \Az_2(T/S)$.  In \S\ref{sec:cores}, we define a
generalization $N_{T/S}$ of the algebra norm functor, which gives a
map in the reverse direction.  Our main result is the following.

\begin{theoremintro}
\label{thm:main}
Let $S$ be a regular integral scheme of dimension $\leq 2$ with 2
invertible and $T \to S$ a finite flat morphism of degree 2 with
regular branch divisor $D \subset S$.  Then the even Clifford algebra
and norm functors
$$
\xymatrix@C60pt{
\Quadrics_2(T/S)
\ar@<1mm>[r]^(.48){\CliffAlg_0} &
\ar@<1mm>[l]^(.52){\Norm_{T/S}} 
\Az_2(T/S)
}
$$
give rise to mutually inverse bijections.
\end{theoremintro}

\pagebreak

This result can be viewed as a significant generalization of the
exceptional isomorphism ${}^2\Dynkin{A}_1=\Dynkin{D}_2$ correspondence
over fields and rings (cf.\ \cite[IV.15.B]{book_of_involutions} and
\cite[\S10]{knus_parimala_sridharan:rank_6_via_pfaffians}) to the
setting of line bundle-valued quadratic forms with simple degeneration
over schemes.  Most of our work goes toward establishing fundamental
local results concerning quadratic forms with simple degeneration (see
\S\ref{sec:preliminary_results}) and the structure of their orthogonal
group schemes, which are nonreductive (see
\S\ref{sec:Orthogonal_groups}).  In particular, we prove that these
group schemes are smooth (see Proposition~\ref{prop:smooth}) and
realize a degeneration of exceptional isomorphisms
${}^2\Dynkin{A}_1=\Dynkin{D}_2$ to $\Dynkin{A}_1=\Dynkin{B}_1$.  We
also establish some structural results concerning quadric surface
bundles over schemes (see \S\ref{sec:simple}) and the formalism of
gluing tensors over surfaces (see \S\ref{sec:gluing}).

Also, we give two surprisingly different applications of our
results. In \S\ref{sec:failure}, we provide a class of quadratic forms
that are counter-examples to the local-global principle for isotropy,
with respect to discrete valuations, over function fields of surfaces
over algebraically closed fields.  Moreover, such forms exist even
over rational function fields, where regular quadratic forms fail to
provide such counter-examples.  In \S\ref{sec:hodge}, using tools from
the theory of moduli of twisted sheaves, we are able to provide a new
proof of the global Torelli theorem for very general cubic fourfolds
containing a plane, a result originally obtained by
Voisin~\cite{voisin:cubic_fourfolds} using Hodge theory.

Our perspective comes from the algebraic theory of quadratic forms.
We employ the even Clifford algebra of a line bundle-valued quadratic
form constructed by Bichsel~\cite{bichsel:thesis}.
Bichsel--Knus~\cite{bichsel_knus:values_line_bundles}, Caenepeel--van
Oystaeyen~\cite{caenepeel_van_oystaeyen} and
Parimala--Sridharan~\cite[{\S
4}]{parimala_sridharan:norms_and_pfaffians} give alternate
constructions, which are all detailed in \cite[\S1.8]{auel:clifford}.
In a similar vein, Kapranov~\cite[\S4.1]{kapranov:derived} (with
further developments by Kuznetsov~\cite[\S3]{kuznetsov:quadrics})
considered the \emph{homogeneous} Clifford algebra of a quadratic
form---the same as the \emph{generalized} Clifford algebra of
\cite{bichsel_knus:values_line_bundles} or the \emph{graded} Clifford
algebra of \cite{caenepeel_van_oystaeyen}---to study the derived
category of projective quadrics and quadric bundles.  We focus on the
even Clifford algebra as a sheaf of algebras, ignoring its geometric
manifestation via the relative Hilbert scheme of lines in the
quadric bundle, as in \cite[\S1]{voisin:cubic_fourfolds} and
\cite[\S5]{hassett_varilly_varilly}.  In this context, we refer to
Hassett--Tschinkel~\cite[\S3]{hassett_tschinkel:spaces_of_sections}
for a version of our result in the case of smooth projective curves
over an algebraically closed field.

Finally, our work on degenerate quadratic forms may also be of
independent interest.  There has been much recent focus on
classification of degenerate (quadratic) forms from various number
theoretic directions.  An approach to Bhargava's \cite{bhargava:ICM}
seminal construction of moduli spaces of ``rings of low rank'' over
arbitrary base schemes is developed by Wood~\cite{wood:binary} where
line bundle-valued degenerate forms (of higher degree) are crucial
ingredients.  In related developments, building on the work of
Delone--Faddeev~\cite{delone_faddeev} over $\Z$ and
Gross--Lucianovic~\cite{gross_lucianovic} over local rings, Venkata
Balaji~\cite{balaji_ternary}, and independently
Voight~\cite{voight:characterizing_quaternion}, used Clifford algebras
of degenerate ternary quadratic forms to classify degenerations of
quaternion algebras over arbitrary bases.  In this context, our main
result can be viewed as a classification of quaternary quadratic forms
with squarefree discriminant in terms of their even Clifford algebras.

\bigskip

{\small\noindent{\bf Acknowledgements.}  %
The first author benefited greatly from a visit at ETH Z\"urich and is
partially supported by National Science Foundation grant MSPRF
DMS-0903039. The second author is partially supported by National
Science Foundation grant DMS-1001872. The authors would specifically
like to thank M.\ Bernardara, J.-L.\ Colliot-Th\'el\`ene, B.\ Conrad,
M.-A.\ Knus, E.\ Macr{\`\i}, and M.\ Ojanguren for many helpful
discussions.}

\section{Reflections on simple degeneration}
\label{sec:simple}

Let $S$ be a noetherian separated integral scheme.  A \linedef{(line
bundle-valued) quadratic form} on $S$ is a triple $(\EE,q,\LL)$, where
$\EE$ is a locally free $\OO_S$-module and $q : \EE \to \LL$ is a
quadratic morphism of sheaves such that the associated morphism of
sheaves $b_q : S^2\EE \to \LL$, defined on sections by $ b_q(v,w) =
q(v+w) - q(v) - q(w), $ is an $\OO_S$-module morphism.  Equivalently,
a quadratic form is an $\OO_S$-module morphism $q : S_2 \EE \to \LL$,
see \cite[Lemma 2.1]{swan:quadric_hypersurfaces} or \cite[Lemma
1.1]{auel:clifford}.  Here, $S^2\EE$ and $S_2\EE$ denote the second
symmetric power and the submodule of symmetric second tensors of
$\EE$, respectively.  There is a canonical isomorphism $S^2(\EE\dual)
\tensor \LL \isom \HHom(S_2\EE,\LL)$.  A line bundle-valued quadratic
form then corresponds to a global section
$$
q \in 
\Gamma(S,\HHom(S_2\EE,\LL)) \isom
\Gamma(S,S^2(\EE\dual)\tensor \LL) \isom
\Gamma_S(\P(\EE),\OO_{\P(\EE)/S}(2)\tensor p\pullback \LL),
$$
where $p : \P(\EE) = \PProj S\graded(\EE\dual) \to S$ and $\Gamma_S$
denotes sections over $S$.  There is a canonical $\OO_S$-module
morphism $\psi_{q} : \EE \to \HHom(\EE,\LL)$ associated to $b_{q}$. A
line bundle-valued quadratic form $(\EE,q,\LL)$ is \linedef{regular}
if $\psi_{q}$ is an $\OO_{S}$-module isomorphism.  Otherwise, the
\linedef{radical} $\rad(\EE,q,\LL)$ is the sheaf kernel of $\psi_q$,
which is a torsion-free subsheaf of $\EE$.  We will mostly dispense
with the adjective ``line bundle-valued.''  We define the
\linedef{rank} of a quadratic form to be the rank of the underlying
module.

A \linedef{similarity} $(\vp,\lambda_{\vp}) : (\EE,q,\LL) \to
(\EE',q',\LL')$ consists of $\OO_S$-module isomorphisms $\vp : \EE \to
\EE'$ and $\lambda_{\vp} : \LL \to \LL'$ such that $q'(\vp(v)) =
\lambda_{\vp} \circ q(v)$ on sections.  A similarity
$(\vp,\lambda_{\vp})$ is an \linedef{isometry} if $\LL=\LL'$ and
$\lambda_{\vp}$ is the identity map.  We write $\similar$ for
similarities and $\isom$ for isometries.  Denote by
$\GOrth(\EE,q,\LL)$ and $\Orth(\EE,q,\LL)$ the presheaves, on the fppf
site $S_{\fl}$, of similitudes and isometries of a quadratic form
$(\EE,q,\LL)$, respectively.  These are sheaves and are representable
by affine group schemes of finite presentation over $S$, indeed closed
subgroupschemes of $\GL(\EE)$.  The similarity factor defines a
homomorphism $\lambda : \GOrth(\EE,q,\LL) \to \Gm$ with kernel
$\Orth(\EE,q,\LL)$.  If $(\EE,q,\LL)$ has even rank $n=2m$, then there
is a homomorphism $\det/\lambda^m : \GOrth(\EE,b,\LL) \to \muu_2$,
whose kernel is denoted by $\GSOrth(\EE,q,\LL)$ (this definition of
$\GSOrth$ assumes 2 is invertible on $S$; in general it is defined as
the kerenel of the Dickson invariant).  The similarity factor $\lambda
: \GSOrth(\EE,q,\LL) \to \Gm$ has kernel denoted by
$\SOrth(\EE,q,\LL)$.  Denote by $\PGOrth(\EE,q,\LL)$ the sheaf
cokernel of the central subgroupschemes $\Gm \to \GOrth(\EE,q,\LL)$ of
homotheties; similarly denote $\PGSOrth(\EE,q,\LL)$.  At every point
where $(\EE,q,\LL)$ is regular, these group schemes are smooth and
reductive (see \cite[II.1.2.6,~III.5.2.3]{demazure_gabriel}) though
not necessarily connected.  In \S\ref{sec:Orthogonal_groups}, we will
study their structure over points where the form is not regular.

The \linedef{quadric bundle} $\pi : Q \to S$ associated to a nonzero
quadratic form $(\EE,q,\LL)$ of rank $n\geq 2$ is the restriction of $p
: \P(\EE) \to S$ via the closed embedding $j : Q \to \P(\EE)$ defined
by the vanishing of the global section $q \in
\Gamma_S(\P(\EE),\OO_{\P(\EE)/S}(2)\tensor p\pullback \LL)$.  Write
$\OO_{Q/S}(1) = j\pullback \OO_{\P(\EE)/S}(1)$.
We say that $(\EE,q,\LL)$ is \linedef{primitive} if $q_x \neq 0$ at
every point $x$ of $S$, i.e., if $q : \EE \to \LL$ is an epimorphism.
If $q$ is primitive then $Q \to \P(\EE)$ has relative codimension 1
over $S$ and $\pi : Q \to S$ is flat of relative dimension $n-2$, cf.\
\cite[8~Thm.~22.6]{matsumura:commutative_ring_theory}.  We say that
$(\EE,q,\LL)$ is \linedef{generically regular} if $q$ is regular over
the generic point of $S$.

Define the \linedef{projective similarity} class of a quadratic form
$(\EE,q,\LL)$ to be the set of similarity classes of quadratic forms
$(\NN \tensor \EE,\id_{\NN^{\tensor 2}}\tensor q,\NN^{\tensor
2}\tensor \LL)$ ranging over all line bundles $\NN$ on $S$.
Equivalently, this is the set of isometry classes $(\NN\tensor\EE,\phi
\circ (\id_{\NN^{\tensor 2}} \tensor q),\LL')$ ranging over all
isomorphisms $\phi : \NN^{\tensor 2}\tensor\LL \to \LL'$ of line
bundles on $S$.  This is referred to as a \linedef{lax-similarity}
class in \cite{balmer_calmes:lax}.  The main result of this section
shows that projectively similar quadratic forms yield isomorphic
quadric bundles, while the converse holds under further hypotheses.

Let $\eta$ be the generic point of $S$ and $\pi : Q \to S$ a quadric
bundle.  Restriction to the generic fiber of $\pi$ gives rise to a
complex
\begin{equation}
\label{eq:minimal}
0 \to \Pic(S) \mapto{\pi\pullback} \Pic(Q) \to \Pic(Q_{\eta}) \to 0
\end{equation}
whose exactness we will study in Proposition~\ref{prop:exactness} below. 

\begin{prop}
\label{prop:proj_sim_quadric}
Let $\pi : Q \to S$ and $\pi' : Q' \to S$ be quadric bundles
associated to quadratic forms $(\EE,q,\LL)$ and $(\EE',q',\LL')$.  If
$(\EE,q,\LL)$ and $(\EE',q',\LL')$ are in the same projective
similarity class then $Q$ and $Q'$ are $S$-isomorphic.  The converse
holds if $q$ is assumed to be generically regular and
\eqref{eq:minimal} is assumed to be exact in the middle.
\end{prop}
\begin{proof}
Let $(\EE,q,\LL)$ and $(\EE',q',\LL')$ be projectively similar with
respect to an invertible $\OO_S$-module $\NN$ and $\OO_S$-module
isomorphisms $\vp : \NN \tensor \EE \to \EE'$ and $\lambda :
\NN^{\tensor 2} \tensor \LL \to \LL'$ preserving quadratic forms.  Let
$p : \P(\EE) \to S$ and $p' : \P(\EE') \to S$ be the associated
projective bundles and $h : \P(\EE') \to \P(\NN\tensor \EE)$ the
$S$-isomorphism associated to to $\vp\dual$.  There is a natural
$S$-isomorphism $g : \P(\NN \tensor \EE) \to \P(\EE)$ satisfying
$g\pullback \OO_{\P(\EE)/S}(1) \isom \OO_{\P(\EE\tensor
\NN)/S}(1)\tensor p'{}\pullback \NN$, see \cite[II Lemma
7.9]{hartshorne:algebraic_geometry}.  Denote by $f = g \circ h :
\P(\EE') \to \P(\EE)$ the composition.  Then we have an equality of
global sections $f\pullback q = q'$ in
$$
\Gamma_S\bigl( \P(\EE'),f\pullback(\OO_{\P(\EE)/S}(2) \tensor
p\pullback \LL)\bigr) 
\isom
\Gamma_S\bigl( \P(\EE'),\OO_{\P(\EE')/S}(2)\tensor p'{}\pullback \LL'\bigr)
$$ 
and so $f$ induces a $S$-isomorphism $Q' \to Q$ upon restriction.

Now suppose that $q$ is generically regular and that
\eqref{eq:minimal} is exact in the middle.  Let $f : Q \to Q'$ be an
$S$-isomorphism.  First, we have a canonical $\OO_S$-module
isomorphism $\EE\dual \isom \pi\pushforward \OO_{Q/S}(1)$.  Indeed,
considering the long exact sequence associated to $p\pushforward$
applied to the short exact sequence
$$
0 \to \OO_{\P(\EE)/S}(-1) \tensor p\pullback \LL\dual \mapto{q}
\OO_{\P(\EE)/S}(1) \to j\pushforward \OO_{Q/S}(1) \to 0
$$
and noting that $R^ip\pushforward\OO_{\P(\EE)/S}(-1)=0$ for $i=0,1$,
we have 
$$
\EE\dual = p\pushforward\OO_{\P(\EE)/S}(1) \isom p\pushforward
j\pushforward \OO_{Q/S}(1) = \pi\pushforward\OO_{Q/S}(1).
$$
Next, we claim that $f\pullback \OO_{Q'/S}(1) \isom \OO_{Q/S}(1)
\tensor \pi\pullback \NN$ for some line bundle $\NN$ on $S$.  Indeed,
over the generic fiber, we have $f\pullback \OO_{Q'/S}(1)_{\eta} =
f_\eta\pullback \OO_{Q'_\eta}(1) \isom \OO_{Q_\eta}(1)$ by the case of
smooth quadrics (since $q$ is generically regular) over fields, cf.\
\cite[Lemma~69.2]{elman_karpenko_merkurjev}.  The exactness of
\eqref{eq:minimal} in the middle finishes the proof of the claim.

Then, note that we have the following chain of $\OO_S$-module
isomorphisms
$$
\EE{}\dual \isom \pi\pushforward \OO_{Q/S}(1) \isom 
\pi\pushforward' f\pushforward (f\pullback\OO_{Q'/S}(1)\tensor\pi\pullback \NN\dual) \isom
\pi\pushforward' \OO_{Q'/S}(1) \tensor \pi\pushforward \pi\pullback \NN\dual
\isom \EE'{}\dual\tensor \NN\dual
$$
(again we need $n \geq 1$). Finally, one is left to check that the
induced dual $\OO_S$-module isomorphism $\NN \tensor \EE' \to \EE$
preserves the quadratic forms.
\end{proof}

\begin{defin}
The determinant $\det \psi_q : \det\EE \to
\det\EE\dual\tensor\LL^{\tensor n}$ gives rise to a global section of
$(\det\EE\dual)^{\tensor 2}\tensor\LL^{\tensor n}$, whose divisor of
zeros is called the \linedef{discriminant divisor} $D$.  The reduced
subscheme associated to $D$ is precisely the locus of points where the
radical of $q$ is nontrivial.  If $q$ is generically regular, then $D
\subset S$ is closed of codimension one.
\end{defin}

\begin{defin}
We say that a quadratic form $(\EE,q,\LL)$ has \linedef{simple
degeneration} if
$$
\rk_{\kappa(x)} \rad(\EE_x,q_x,\LL_x) \leq 1
$$ 
for every closed point $x$ of $S$, where $\kappa(x)$ is the residue
field of $\OO_{S,x}$.
\end{defin}

Our first lemma concerns the local structure of 
simple degeneration.

\begin{lemma}
\label{lem:simple_local}
Let $(\EE,q)$ be a quadratic form with simple degeneration over the
spectrum of a local ring $R$ with 2 invertible.  Then $(\EE,q) \isom
(\EE_1,q_1) \perp (R,\qf{\pi})$ where $(\EE_1,q_1)$ is regular and
$\pi \in R$.
\end{lemma}
\begin{proof}
Over the residue field $k$, the form $(\EE,q)$ has a regular
subform $(\ol{\EE}_1,\ol{q}_1)$ of corank one, which can be lifted to a
regular orthogonal direct summand $(\EE_1,q_1)$ of corank 1 of
$(\EE,q)$, cf.\ \cite[Cor.~3.4]{baeza:semilocal_rings}.  This gives
the required decomposition.
Moreover, we can lift a diagonalization
$\ol{q}_1\isom\qf{\ol{u}_1,\dotsc,\ol{u}_{n-1}}$ with $\ol{u}_i \in
k\mult$, to a diagonalization 
$$
q \isom \qf{u_1,\dotsc,u_{n-1},\pi},
$$
with $u_i \in R\mult$ and $\pi \in R$.
\end{proof}

Let $D \subset S$ be a regular divisor.  Since $S$ is normal, the
local ring $\OO_{S,D'}$ at the generic point of a component $D'$ of
$D$ is a discrete valuation ring.  When 2 is invertible on $S$,
Lemma~\ref{lem:simple_local} shows that a quadratic form $(\EE,q,\LL)$
with simple degeneration along $D$ can be diagonalized over
$\OO_{S,D'}$ as
$$
q \isom \qf{ u_1, \dotsc, u_{r-1}, u_r \pi^e }
$$
where $u_i$ are units and $\pi$ is a parameter of $\OO_{S,D'}$.  We
call $e \geq 1$ the \linedef{multiplicity} of the simple degeneration
along $D'$.  If $e$ is even for every component of $D$, then there is
a birational morphism $g : S' \to S$ such that the pullback of
$(\EE,q,\LL)$ to $S'$ is regular.  We will focus on quadratic forms
with simple degeneration of multiplicity one along (all components of)
$D$.

We can give a geometric interpretation of having simple degenerate.

\begin{prop}
\label{prop:isolated}
Let $\pi : Q \to S$ be the quadric bundle associated to a generically
regular quadratic form $(\EE,q,\LL)$ over $S$ and $D \subset S$ its
discriminant divisor.  Then:
\begin{enumerate}
\item $q$ has simple degeneration if and only if the fiber $Q_x$ of
its associated quadric bundle has at worst isolated singularities for
each closed point $x$ of $S$;

\item if 2 is invertible on $S$ and $D$ is reduced, then any simple
degeneration along $D$ has multiplicity one;

\item if 2 is invertible on $S$ and $D$ is regular, then any
degeneration along $D$ is simple of multiplicity one;

\item if $S$ is regular and $q$ has simple degeneration, then $D$ is
regular if and only if $Q$ is regular.
\end{enumerate}
\end{prop}
\begin{proof}
The first claim follows from the classical geometry of quadrics over a
field:\ the quadric of a nondegenerate form is smooth while the
quadric of a form with nontrivial radical has isolated singularity if
and only if the radical has rank one.  As for the second claim, the
multiplicity of the simple degeneration is exactly the
scheme-theoretic multiplicity of the divisor $D$.  For the third
claim, see \cite[\S3]{colliot_skorobogatov:quadriques},
\cite[Rem.~7.1]{hassett_varilly_varilly}, or
\cite[Rem.~2.6]{auel_bernardara_bolognesi:quadrics}.  The final claim
is standard, cf.\ \cite[Lemma~5.2]{hassett_varilly_varilly}.
\end{proof}

We do not need the full flexibility of the following general result,
but we include it for completeness.

\begin{prop}
\label{prop:exactness}
Let $\pi : Q \to S$ be a flat morphism of noetherian integral
separated normal schemes and $\eta$ be the generic point of $S$.  Then
the complex \eqref{eq:minimal} is:
\begin{enumerate}
\item exact at right if $Q$ is locally factorial;
\item exact in the middle if $S$ is locally factorial;
\item exact at left if $\pi : Q \to S$ is proper with geometrically
integral fibers.
\end{enumerate}
\end{prop}
\begin{proof}
First, note that flat pullback and restriction to the generic fiber
give rise to an exact sequence of Weil divisor groups
\begin{equation}
\label{eq:div}
0 \to \Div(S) \mapto{\pi\pullback} \Div(Q) \to \Div(Q_{\eta}) \to 0.
\end{equation}
Indeed, as $\Div(Q_\eta) = \varinjlim \Div(Q_U)$, where the limit is
taken over all dense open sets $U \subset S$ and we write $Q_U = Q
\times_S U$, the exactness at right and center of sequence
\eqref{eq:div} then follows from usual exactness of the excision
sequence
$$
Z^0(\pi\inv(S\bslash U)) \to \Div(Q) \to \Div(Q_U) \to 0
$$
cf.\ \cite[1~Prop.~1.8]{fulton:intersection_theory}.  Finally,
sequence \eqref{eq:div} is exact at left since $\pi$ is surjective on
codimension 1 points.

Since $\pi$ is dominant, the sequence of Weil divisor groups
induces an exact sequence
of Weil divisor class groups, which is the bottom row of the following
commutative diagram
$$
\xymatrix@R=14pt{
\Pic(S) \ar@{^{(}->}[d] \ar[r]^{\pi\pullback} & \Pic(Q) \ar@{^{(}->}[d] \ar[r]
&\Pic(Q_{\eta}) \ar@{^{(}->}[d] \ar[r] & 0\\
\Cl(S) \ar[r]^{\pi\pullback} & \Cl(Q) \ar[r] & \Cl(Q_{\eta}) \ar[r] & 0
}
$$
of abelian groups.
The vertical inclusions become equalities
under a locally factorial hypothesis, cf.\ \cite[Cor.~21.6.10]{EGA4}.
Thus \textit{a)} and \textit{b)} are immediate consequences of diagram
chases.

To prove \textit{c)}, assume $\pi\pullback[\LL] = [\pi\pullback\LL] =
0$ in $\Cl(Q)$ for the class $[\LL] \in \Cl(S)$ of some line bundle
$\LL \in \Pic(Q)$.  Then $[\pi\pullback\LL] = \divisor_Y(f)$ in
$\Div(Q)$ for some $f \in K_Q\mult$.  But as $[\pi\pullback\LL]_\eta =
0$ in $\Div(Q_\eta)$, we have that $\divisor_{Q_\eta}(f)=0$ (since
$K_Q=K_{Q_\eta}$) so that $f \in \Gamma(Q_{\eta},\OO_{Q_\eta}\mult)$,
i.e., $f$ has neither zeros nor poles.  By the
hypothesis on $\pi$, we have that $Q_\eta$ is a proper geometrically
integral $K_S$-scheme, so that
$\Gamma(Q_{\eta},\OO_{Q_\eta}\mult)=K_S$.  In particular, $f \in K_S$
via the inclusion $K_S \subsetto K_Q$.  Hence
$\pi\pullback([\LL]-\divisor_S(f))=0$ in $\Div(Q)$, thus $[\LL] =
\divisor_S(f)$ in $\Div(S)$, and so $\LL$ is trivial in $\Pic(S)$.
\end{proof}

\begin{cor}
\label{cor:simple_minimal}
Let $S$ be a regular integral scheme with 2 invertible and
$(\EE,q,\LL)$ a quadratic form on $S$ of rank $\geq 4$ having at most
simple degeneration along a regular divisor $D \subset S$.  Let $\pi :
Q \to S$ be the 
associated quadric bundle. Then the complex \eqref{eq:minimal} is exact.
\end{cor}
\begin{proof}
First, recall that a quadratic form over a field contains a
nondegenerate subform of rank $\geq 3$ if and only if its associated
quadric is irreducible, cf.\
\cite[I~Ex.~5.12]{hartshorne:algebraic_geometry}.  Hence the fibers of
$\pi$ are geometrically irreducible.  By Proposition
\ref{prop:isolated}, $Q$ is regular.  Quadratic forms with simple
degeneration are primitive, hence $\pi$ is flat.  Thus we can apply
all the parts of
Proposition~\ref{prop:exactness}.
\end{proof}

We will define $\Quadrics_n^D(S)$ to be the set of projective
similarity classes of line bundle-valued quadratic forms of rank $n+2$
on $S$ with simple degeneration of multiplicity one along an effective
Cartier divisor $D$.  An immediate consequence of Propositions
\ref{prop:proj_sim_quadric} and \ref{prop:isolated} and Corollary
\ref{cor:simple_minimal} is the following.

\begin{cor}
For $n \geq 2$ and $D$ reduced, the set $\Quadrics_n^D(S)$ is in
bijection with the set of $S$-isomorphism classes of quadric bundles
of relative dimension $n$ with isolated singularities in the fibers
above $D$.
\end{cor}

\begin{defin}
\label{defin:discriminant_cover}
Now let $(\EE,q,\LL)$ be a quadratic form of rank $n$,
$\CliffAlg_0=\CliffAlg_0(\EE,q,\LL)$ its even Clifford algebra (see
\cite{bichsel_knus:values_line_bundles} or
\cite[\S1.8]{auel:clifford}), and $\CliffZ = \CliffZ(\EE,q,\LL)$ its
center.  Then $\CliffAlg_0$ is a locally free $\OO_S$-algebra of rank
$2^{n-1}$, cf.\ \cite[IV.1.6]{knus:quadratic_hermitian_forms}.  If $q$ is generically regular of even rank (we are still
assuming that $S$ is integral and regular) then $\CliffZ$ is a locally
free $\OO_S$-algebra of rank two, see
\cite[IV~Prop.~4.8.3]{knus:quadratic_hermitian_forms}.  The associated
finite flat morphism $f : T \to S$ of degree two is called the
\linedef{discriminant cover}.
\end{defin}

\begin{lemma}[{\cite[App.~B]{auel_bernardara_bolognesi:quadrics}}]
\label{lem:disc_ram}
Let $(\EE,q,\LL)$ be a quadratic form of even rank with simple
degeneration of multiplicity one along $D \subset S$ and $f : T \to S$
its discriminant cover.  Then $f\pullback \OO(D)$ is a square in
$\Pic(T)$ and the branch divisor of $f$ is precisely $D$.
\end{lemma}

By abuse of notation, we also denote by $\CliffAlg_0 =
\CliffAlg_0(\EE,q,\LL)$ the $\OO_T$-algebra associated to the
$\CliffZ$-algebra $\CliffAlg_0 = \CliffAlg(\EE,q,\LL)$.  The center
$\CliffZ$ is an \'etale algebra over every point of $S$ where
$(\EE,q,\LL)$ is regular and $\CliffB_0$ is an Azumaya algebra over
every point of $T$ lying over a point of $S$ where $(\EE,q,\LL)$ is
regular.  

\begin{lemma}
\label{lem:locally_free}
Let $(\EE,q,\LL)$ be a quadratic form with simple degeneration over an
integral scheme $S$ with 2 invertible and let $T \to S$ be the
discriminant cover.  Then $\CliffB_0$ is a locally free
$\OO_T$-algebra.
\end{lemma}
\begin{proof}
The question is local, so we can assume that $S = \Spec R$ for a local
domain $R$ with 2 invertible.  We fix a trivialization of $\LL$ and
let $n$ be the rank of $q$.  By Lemma~\ref{lem:simple_local}, we write
$(\EE,q) = (\EE_1,q_1) \perp (R,\qf{\pi})$ for $q_1$ regular and $\pi
\in R$.  Consider the canonically induced homomorphism
$\CliffAlg_0(\EE_1,q_1) \to \CliffAlg_0(\EE,q)$.  We claim that the
map
\begin{equation}
\label{eq:CliffB_locally_free}
\CliffAlg_0(\EE_1,q_1) \tensor_{\OO_S} \CliffZ(\EE,q) \to \CliffAlg_0(\EE,q)
\end{equation}
induced from multiplication in $\CliffAlg_0(\EE,q)$, is a
$\CliffZ(\EE,q)$-algebra isomorphism.  Indeed, since
$\CliffAlg_0(\EE_1,q_1)$ is an Azumaya $\OO_S$-algebra, then
$\CliffAlg_0(\EE_1,q_1) \tensor_{\OO_S} \CliffZ(\EE,q)$ is an Azumaya
$\CliffZ(\EE,q)$-algebra.  In particular, the map is injective.  If
$q$ is generically regular (i.e., $\pi\neq 0$), then then this map is
generically an isomorphism, hence an isomorphism.  Otherwise $\pi=0$,
in which case $\CliffZ(\EE,q) \isom \OO_S[\ep]/(\ep^2)$ and we can argue
directly using the exact sequence
$$
0 \to e \CliffAlg_1(\EE_1,q_1) \to \CliffAlg_0(\EE,q) \to
\CliffAlg_0(\EE_1,q_1) \to 0
$$
where $e \in \EE$ generates the radical, and 
the fact that
$e\CliffAlg_1(\EE_1,q_1) \isom \epsilon \CliffAlg_0(\EE_1,q_1)$.
\end{proof}

Finally, we prove a strengthened version of
\cite[Prop.~3.13]{kuznetsov:quadrics}.

\begin{lemma}
\label{lem:simple_field}
Let $(V,q)$ be a quadratic form of even rank $n=2m > 2$ over a field $k$.
Then the following are equivalent:
\begin{enumerate}
\item The radical of $q$ has rank at most 1.
\item The center $Z(q) \subset C_0(q)$ is a $k$-algebra of rank 2.
\item The algebra $C_0(q)$ is $Z(q)$-Azumaya of degree $2^{m-1}$.
\end{enumerate}
If $n=2$, then $C_0(q)$ is always commutative.
\end{lemma}
\begin{proof}
We will prove that
\textit{a)}~$\Rightarrow$~\textit{c)}~$\Rightarrow$~\textit{b)}~$\Rightarrow$~\textit{a)}.
If $q$ is nondegenerate (i.e., has trivial radical), then it is
classical that $Z(q)$ is an \'etale quadratic algebra and $C_0(q)$ is
an Azumaya $Z(q)$-algebra.  If $\rad(q)$ has rank 1, generated by $v
\in V$, then a straightforward computation shows that $Z(q) \isom
k[\varepsilon]/(\epsilon^2)$, where $\epsilon \in v C_1(q) \cap Z(q)
\bslash k$.  Furthermore, we have that $C_0(q)
\tensor_{k[\varepsilon]/(\varepsilon^2)} k \isom C_0(q)/v C_1(q) \isom
C_0(q/\rad(q))$ where $q/\rad(q)$ is nondegenerate of rank $n-1$, cf.\
\cite[II~\S11,~p.~58]{elman_karpenko_merkurjev}.  Since $C_0(q)$ is
finitely generated and free as a $Z(q)$-module by
Lemma~\ref{lem:locally_free}, and has special fiber a central
simple algebra $C_0(q/\rad(q))$ of degree $2^{m-1}$, it is
$Z(q)$-Azumaya of degree $2^{m-1}$.  Hence we've proved that
\textit{a)}~$\Rightarrow$~\textit{c)}

The fact that \textit{c)}~$\Rightarrow$~\textit{b)} is clear from a
dimension count.  To prove \textit{b)}~$\Rightarrow$~\textit{a)},
suppose that $\rk_k \rad(q) \geq 2$.  Then the embedding $\bigwedge^2
\rad(q) \subset C_0(q)$ is central (and does not contain the central
subalgebra generated by $V^{\tensor n}$, as $q$ has rank $>2$).  More
explicitly, if $e_1, e_2, \dotsc, e_n$ is a block diagonalization
(into 1 and 2 dimensional spaces), then $k \oplus k e_1\dotsm e_n
\oplus \exterior^2 \rad(q) \subset Z(q)$.  Thus $Z(q)$ has $k$-rank at
least $2 + \rk_k \bigwedge^2 \rad(q)\geq
3$.
\end{proof}

\begin{prop}
\label{prop:simple_azumaya}
Let $(\EE,q,\LL)$ be a generically regular quadratic form of even rank
on $S$ with discriminant cover $f : T \to S$.  Then
$\CliffB_0(\EE,q,\LL)$ is an Azumaya $\OO_T$-algebra if and only if
$(\EE,q,\LL)$ has simple degeneration or has rank 2 (and any
degeneration).
\end{prop}
\begin{proof}
Since $\CliffB_0$ is a locally free $\OO_S$-module it is a locally
free $\OO_T$-module (as $\OO_T$ is locally free of rank 2).  Thus
$\CliffB_0$ is an $\OO_T$-Azumaya algebra if and only if its fiber at
every closed point $y$ of $T$ is a central simple $\kappa(y)$-algebra.
If $f(y)=x$, then $\kappa(y)$ is a $\kappa(x)$-algebra of rank 2.
Then we apply Lemma~\ref{lem:simple_field} to the fiber of $\CliffB_0$
at $x$.
\end{proof}

\section{Orthogonal groups with simple degeneration}
\label{sec:Orthogonal_groups}

The main results of this section concern the special (projective)
orthogonal group schemes of quadratic forms with simple degeneration
over semilocal principal ideal domains.  Let $S$ be a regular integral
scheme.  Recall, from Proposition~\ref{prop:simple_azumaya}, that if
$(\EE,q,\LL)$ is a line bundle-valued quadratic form on $S$ with
simple degeneration along a closed subscheme $D$ of codimension 1,
then the even Clifford algebra $\CliffAlg_0(q)$ is an Azumaya algebra
over the discriminant cover $T \to S$.

\begin{theorem}
\label{thm:isom}
Let $S$ be a regular scheme with 2 invertible, $D$ a regular
divisor, $(\EE,q,\LL)$ a quadratic form of rank 4 on $S$ with
simple degeneration along $D$, $T \to S$ its discriminant cover, and
$\CliffB_0(q)$ its even Clifford algebra over $T$.  The canonical
homomorphism
$$
c: \PGSOrth(q) \to R_{T/S}\PGL(\CliffB_0(q)),
$$
induced from the functor $\CliffB_0$, is an isomorphism of $S$-group schemes.
\end{theorem}

The proof involves several preliminary general results concerning
orthogonal groups of quadratic forms with simple degeneration and will
occupy the remainder of this section.

\medskip

Let $S = \Spec R$ be an affine scheme with 2 invertible, $D \subset S$
be the closed scheme defined by an element $\pi$ in the
Jacobson radical of $R$, and let $(V,q) = (V_1,q_1) \perp (R,\qf{ \pi
})$ be a quadratic form of rank $n$ over $S$ with $q_1$ regular and
$V_1$ free.  Let $Q_1$ be a Gram matrix of $q_1$.  Then as an
$S$-group scheme, $\Orth(q)$ is the subvariety of the affine space of
block matrices
\begin{equation}
\label{eq:relations}
\begin{pmatrix} 
A & v \\
w & u 
\end{pmatrix}
\quad \mbox{satisfying} \quad
\begin{minipage}{6cm}
$A\transpose Q_1 A + \pi\, w\transpose w = Q_1$ \\
$A\transpose Q_1 v + u\pi\, w\transpose = 0$ \\
$v\transpose Q_1 v = (1-u^2)\pi$
\end{minipage}
\end{equation} 
where $A$ is an invertible $(n-1)\times (n-1)$ matrix, $v$ is an $n
\times 1$ column vector, $w$ is a $1 \times n$ row vector, and $u$ a
unit. Note that since $A$ and $Q_1$ are invertible, the second
relation in \eqref{eq:relations} implies that $v$ is determined by $w$
and $u$ and that $\obar{v} = 0$ over $R/\pi$, and hence the third
relation implies that $\obar{u}^2 = 1$ in $R/\pi$.  Define $\SOrth(q)
= \ker(\det : \Orth(q) \to \Gm)$.  If $R$ is an integral domain then
$\det$ factors through $\muu_2$ and $\SOrth(q)$ is the irreducible
component of the identitiy.

\begin{prop}
\label{prop:smooth_local}
Let $R$ be a regular local ring with 2 invertible, $\pi \in \mm$
in the maximal ideal, and $(V,q) = (V_1,q_1) \perp (R,\qf{\pi})$ a quadratic form with $q_1$
regular of rank $n-1$ of $R$. Then $\Orth(q)$ and $\SOrth(q)$ are
smooth $R$-group schemes.
\end{prop}
\begin{proof}
Let $K$ be the fraction field of $R$ and $k$ its residue field.
First, we'll show that the equations in \eqref{eq:relations} define a
local complete intersection morphism in the affine space $\A_R^{n^2}$
of $n\times n$ matrices over $R$.  Indeed, the condition that the
generic $n \times n$ matrix $M$ over $R[x_1,\dotsc,x_{n^2}]$ is
orthogonal with respect to a given symmetric $n\times n$ matrix $Q$
over $R$ can be written as the equality of symmetric matrices
$M\transpose Q M = Q$ over $R[x_1,\dotsc,x_{n^2}][(\det M)\inv]$,
hence giving $n(n+1)/2$ equations.  Hence, the orthogonal group is the
scheme defined by these $n(n+1)/2$ equations in the Zariski open of
$\A_R^{n^2}$ defined by $\det M$.

Since $q$ is generically regular of rank $n$, the generic fiber of
$\Orth(q)$ has dimension $n(n-1)/2$.  By \eqref{eq:relations}, the
special fiber of $\SOrth(q)$ is isomorphic to the group scheme of
euclidean transformations of the regular quadratic space $(V_1,q_1)$,
which is the semidirect product
\begin{equation}
\label{eq:special_fiber}
\SOrth(q) \times_{R} k \isom \Ga^{n-1} \rtimes \Orth(q_{1,k})
\end{equation}
where $\Ga^{n-1}$ acts in $V_1$ by translation and $\Orth(q_{1,k})$
acts on $\Ga^{n-1}$ by conjugation.  In particular, the special fiber
of $\SOrth(q)$ has dimension $(n-1)(n-2)/2 + (n-1) = n(n-1)/2$, and
similarly with $\Orth(q)$.

In particular, $\Orth(q)$ is a local complete intersection morphism.
Since $R$ is Cohen--Macaulay (being regular local) then
$R[x_1,\dotsc,x_{n^2}][(\det M)\inv]$ is Cohen--Macaulay, and thus
$\Orth(q)$ is Cohen--Macaulay.  By the ``miracle flatness'' theorem,
equidimensional and Cohen--Macaulay over a regular base implies that
$\Orth(q) \to \Spec R$ is flat, cf.\ \cite[Prop.~15.4.2]{EGA4} or
\cite[8~Thm.~23.1]{matsumura:commutative_ring_theory}.  Thus
$\SOrth(q) \to \Spec R$ is also flat.  The generic fiber of
$\SOrth(q)$ is smooth since $q$ is generically regular while the
special fiber is smooth since it is a (semi)direct product of smooth
schemes (recall that $\Orth(q_1)$ is smooth since 2 is invertible).
Hence $\SOrth(q) \to \Spec R$ flat and has geometrically smooth
fibers, hence is smooth.
\end{proof}

\begin{prop}
\label{prop:smooth}
Let $S$ be a regular scheme with 2 invertible and $(\EE,q,\LL)$ a
quadratic form of even rank on $S$ with simple degeneration.  Then the
group schemes $\Orth(q)$, $\SOrth(q)$, $\GOrth(q)$, $\GSOrth(q)$, $\PGOrth(q)$, and
$\PGSOrth(q)$ are $S$-smooth.  If $T \to S$
is the discriminant cover and $\CliffB_0(q)$ is the even Clifford algebra
of $(\EE,q,\LL)$ over $T$, then $R_{T/S}\GL_1(\CliffB_0(q))$,
$R_{T/S}\SL_1(\CliffB_0(q))$, and $R_{T/S}\PGL_1(\CliffB_0(q))$ are smooth $S$-schemes.
\end{prop}
\begin{proof}
The $S$-smoothness of $\Orth(q)$ and $\SOrth(q)$ follows from the
fibral criterion for smoothness, with
Proposition~\ref{prop:smooth_local} handling points of $S$ contained
in the discrimnant divisor.  As $\GOrth \isom
(\Orth(q)\times\Gm)/\muu_2$, $\GSOrth(q) \isom
(\SOrth(q)\times\Gm)/\muu_2$, $\PGOrth(q) \isom \GOrth(q)/\Gm$,
$\PGSOrth(q) \isom \GSOrth(q)/\Gm$ are quotients of $S$-smooth group
schemes by flat closed subgroups, they are $S$-smooth.  Finally,
$\CliffB_0(q)$ is an Azumaya $\OO_T$-algebra by Proposition
\ref{prop:simple_azumaya}, hence $\GL_1(\CliffB_0(q))$,
$\SL_1(\CliffB_0(q))$, and $\PGL_1(\CliffB_0(q))$ are smooth
$T$-schemes, hence their Weil restrictions via the finite flat map
$T\to S$ are $S$-smooth by
\cite[App.~A.5,~Prop.~A.5.2]{conrad_gabber_prasad}.
\end{proof}

\begin{remark}
If the radical of $q_s$ has rank $\geq 2$ at a point $s$ of $S$, a
calculation shows that the fiber of $\Orth(q) \to S$ over $s$ has
dimension $> n(n-1)/2$.  In particular, if $q$ is generically regular
over $S$ then $\Orth(q) \to S$ is not flat.  The smoothness of
$\Orth(q)$ is a special feature of quadratic forms with simple
degeneration.
\end{remark}

We will also make frequent reference to the classical version of
Theorem~\ref{thm:isom} in the regular case, when the discriminant
cover is \'etale.

\begin{theorem}
\label{thm:etale}
Let $S$ be a scheme and $(\EE,q,\LL)$ a regular quadratic form of rank
4 with discriminant cover $T \to S$ and even Clifford algebra
$\CliffB_0(q)$ over $T$.  The canonical homomorphism
$$
c : \PGSOrth(q) \to R_{T/S}\PGL(\CliffB_0(q)),
$$
induced from the functor $\CliffB_0$, is an isomorphism of $S$-group
schemes.
\end{theorem}
\begin{proof}
The proof over affine schemes $S$ in
\cite[\S10]{knus_parimala_sridharan:rank_6_via_pfaffians} carries over
immediately.  See \cite[IV.15.B]{book_of_involutions} for the
particular case of $S$ the spectrum of a field.  Also see
\cite[\S5.3]{auel:clifford}.
\end{proof}

Finally, we come to the proof of the main result of this section.

\begin{proof}[Proof of Theorem~\ref{thm:isom}]
We will use the following fibral criteria for relative isomorphisms
(cf.\ \cite[IV.4~Cor.~17.9.5]{EGA4}):\ let $g : X \to Y$ be a morphism
of $S$-schemes locally of finite presentation over a scheme $S$ and
assume $X$ is $S$-flat, then $g$ is an $S$-isomorphism if and only if
its fiber $g_s : X_s \to Y_s$ is an isomorphism over each geometric
point $s$ of $S$.

For each $s$ in $S \bslash D$, the fiber $q_s$ is a regular quadratic
form over $\kappa(s)$, hence the fiber $c_s : \PGSOrth(q_s) \to
R_{T/S}\PGL(\CliffB_0(q_s))$ is an isomorphism by
Theorem~\ref{thm:etale}.  We are thus reduced to considering the
geometric fibers over points in $D$.  Let $s=\Spec k$ be a geometric
point of $D$.  
By Lemma~\ref{lem:simple_field}, there is a natural identification of
the fiber $T_s = \Spec k_\ep$, where $k_\ep=k[\ep]/(\ep^2)$.

We use the following criteria for isomorphisms of group schemes (cf.\
\cite[VI~Prop.~22.5]{book_of_involutions}):\ let $g : X \to Y$ be a
homomorphism of affine $k$-group schemes of finite type over an
algebraically closed field $k$ and assume that $Y$ is smooth, then $g$
is a $k$-isomorphism if and only if $g : X(k) \to Y(k)$ is an isomorphism
on $k$-points and the Lie algebra map $\text{d}g : \text{Lie}(X) \to
\text{Lie}(Y)$ is an injective map of $k$-vector spaces.

First, we shall prove that $c$ is an isomorphism on $k$ points.
Applying cohomology to the exact sequence
$$
1 \to \muu_2 \to \SOrth(q) \to \PGSOrth(q) \to 1,
$$
we see that the corresponding sequence of $k$-points
is exact since $k$ is algebraically closed.  Hence it suffices to show
that $\SOrth(q)(k) \to \PGL_1(\CliffB_0(q))(k)$ is surjective with
kernel $\muu_2(k)$.

Write $q = q_1 \perp \qf{0}$, where $q_1$ is regular over $k$.  Denote
by $\Group{E}$ the unipotent radical of $\SOrth(q)$.  We will now
proceed to define the following diagram
$$
\xymatrix@R=16pt{
&  & 1 \ar[d] & 1 \ar[d] \\
&  & \muu_2 \ar[d] \ar@{=}[r] & \muu_2 \ar[d] \\
1 \ar[r] & \Group{E} \ar@{..>}[d] \ar[r] & \SOrth(q) \ar[d] \ar[r]&
\Orth(q_1) \ar[d] \ar[r] & 1 \\
1 \ar[r] & I+\ep\, \Lie{c}_0(q) \ar[r] & \PGL_1(\CliffB_0(q))  \ar[d] \ar[r] &
\PGL_1(\CliffB_0(q_1)) \ar[d]\ar[r] & 1 \\
&  & 1 & 1 \\
}
$$
of groups schemes over $k$, and verify that it is commutative with
exact rows and columns.  This will finish the proof of the statement
concerning $c$ being an isomorphism on $k$-points.  We have
$\Het^1(k,\Group{E})=0$ and also $\Het^1(k,\muu_2)=0$, as $k$ is
algebraically closed.  Hence it suffices to argue after taking
$k$-points in the diagram.

The central and right most vertical columns are induced by the
standard action of the (special) orthogonal group on the even Clifford
algebra.  The right most column is an exact sequence
$$
1 \to \muu_2 \to \Orth(q_1) \isom \muu_2 \times \SOrth(q_1) \to \PGL_1(\CliffB_0(q_1)) \to 1
$$
arising from the split isogeny of type $\Dynkin{A}_1=\Dynkin{B}_1$,
cf.\ \cite[IV.15.A]{book_of_involutions}.
The central row is defined by the map
$\SOrth(q)(k) \to \Orth(q)(k)$ defined by
$$
\begin{pmatrix} 
A & v \\
w & u 
\end{pmatrix}
\mapsto A
$$ 
in the notation of \eqref{eq:relations}.  In particular, the group $\Group{E}(k)$ consists of block matrices
of the form
$$
\begin{pmatrix}
I & 0 \\
w & 1
\end{pmatrix}
$$
for $w \in \A^3(k)$.
 Since $\Orth(q_1)$ is
semisimple, the kernel contains the unipotent radical $\Group{E}$, so
coincides with it by a dimension count.  The bottom row is defined as
follows.  By \eqref{eq:CliffB_locally_free}, we have $\CliffB_0(q)
\isom \CliffB_0(q_1) \tensor_k \CliffZ(q) \isom \CliffB_0(q_1)
\tensor_k k_\ep$.  The map $\PGL_1(\CliffB_0(q)) \to
\PGL_1(\CliffB_0(q_1)$ is thus defined by the reduction $k_\ep \to k$.
This also identifies the kernel as $I + \ep\, \Lie{c}_0(q)$, where
$\Lie{c}_0(q)$ is the affine scheme of reduced trace zero elements of
$\CliffB_0(q)$, which is identified with the Lie algebra of
$\PGL_1(\CliffB_0(q))$ in the usual way.  The only thing to check is
that the bottom left square commutes (since by
\eqref{eq:special_fiber}, the central row is split).  By the five
lemma, it will then suffice to show that $\Group{E}(k) \to 1 + \ep\,
\Lie{c}_0(q)(k)$ is an isomorphism.

To this end, we can diagonalize $q = \qf{1,-1,1,0}$, since $k$ is
algebraically closed of characteristic $\neq 2$.  Let $e_1, \dotsc,
e_4$ be the corresponding orthogonal basis.  Then $\CliffB_0(q_1)(k)$
is generated over $k$ by $1$, $e_1e_2$, $e_2e_3$, and $e_1e_3$ and we
have an identification $\vp : \CliffB_0(q_1)(k) \to M_2(k)$ given by
$$
1 \mapsto
\begin{pmatrix}
1 & 0 \\
0 & 1 
\end{pmatrix},
\quad
e_1e_2 \mapsto
\begin{pmatrix}
0 & 1 \\
1 & 0 
\end{pmatrix},
\quad
e_2e_3 \mapsto
\begin{pmatrix}
1 & 0 \\
0 & -1 
\end{pmatrix},
\quad
e_1e_3 \mapsto
\begin{pmatrix}
0 & 1 \\
-1 & 0 
\end{pmatrix}.
$$
Similarly, $\CliffB_0(q)$ is generated over $\CliffZ(q)=k_\ep$ by $1$,
$e_1e_2$, $e_2e_3$, and $e_1e_3$, since we have
$$
e_1e_4 = \ep\, e_2e_3, \quad
e_2e_4 = \ep\, e_1e_3, \quad
e_3e_4 = \ep\, e_1e_2, \quad
e_1e_2e_3e_4 = \ep.
$$
and we have a identification $\psi : \CliffB_0(q) \to M_2(k_\ep)$
extending $\vp$.  With respect to this $k_\ep$-algebra isomorphsm, we
have a group isomorphism $\PGL_1(\CliffB_0(q))(k) = \PGL_2(k_\ep)$ and
a Lie algebra isomorphism $\Lie{c}_0(q)(k) \isom \Lie{sl}_2(k)$, where
$\Lie{sl}_2$ is the scheme of traceless $2 \times 2$ matrices.
We claim that the map $\Group{E}(k) \to I+\ep\, \Lie{sl}_2(k)$ is
explicitly given by
\begin{equation}
\label{eq:explicit}
\begin{pmatrix}
1 & 0 & 0 & 0 \\
0 & 1 & 0 & 0 \\
0 & 0 & 1 & 0 \\
a & b & c & 1
\end{pmatrix}
\mapsto
I - \frac{1}{2} \ep
\begin{pmatrix}
a & -b+c \\
b+c & -a
\end{pmatrix}. 
\end{equation}
Indeed, let $\phi_{a,b,c} \in \Group{E}(S_0)$ be the orthogonal transformation
whose matrix displayed in \eqref{eq:explicit}, and $\sig_{a,b,c}$ its
image in $I+\ep\, \Lie{sl}_2(k)$, thought of as an automorphism of $\CliffB_0(q)(k_\ep)$.
Then we have
\begin{align*}
\sig_{a,b,c}(e_1e_2) & = e_1e_2 + b\ep \,e_2e_3 - a\ep\, e_1e_3 \\
\sig_{a,b,c}(e_2e_3) & = e_2e_3 + c\ep \,e_1e_3 - b\ep\, e_1e_2 \\
\sig_{a,b,c}(e_1e_3) & = e_1e_3 + c\ep \,e_2e_3 - a\ep\, e_1e_2
\end{align*}
and $\sig_{a,b,c}(\ep) = \ep$.  It is then a straightforward calculation to
see that 
$$
\sig_{a,b,c} = \mathrm{ad}\bigl( 1 - \frac{1}{2}\ep(c\, e_1e_2 + a\,
e_2e_3 - b\, e_1e_3) \bigr),
$$
where $\mathrm{ad}$ is conjugation in the Clifford algebra, and
furthermore, that $\psi$ takes $c\, e_1e_2 + a\, e_2e_3 - b\, e_1e_3$
to the $2\times 2$ matrix displayed in \eqref{eq:explicit}.  Thus the
map $\Group{E}(k) \to I+\ep\, \Lie{sl}_2(k)$ is as stated, and in
particular, is an isomorphism.  Thus the diagram is commutative with
exact rows and columns, and in particular, $c : \PGSOrth(q) \to
\PGL_1(\CliffB_0(q))$ is an isomorphism on $k$-points.

Now we prove that the Lie algebra map $\text{d}c$ is injective. 
Consider the commutative diagram
$$
\xymatrix{
1 \ar[r] & I+ x \, \Lie{so}(q)(k) \ar[d]^{1+x\, \text{d}c} \ar[r] & \SOrth(q)(k[x]/(x^2)) \ar[d]^{c(k[x]/(x^2))} \ar[r]&
\SOrth(q)(k) \ar[d] \ar[r] & 1 \\
1 \ar[r] & I+ x\, \Lie{g}(k) \ar[r] & \PGL_1(\CliffB_0(q))(k[\ep,x]/(\ep^2,x^2)) \ar[r] &
\PGL_1(\CliffB_0(q))(k_\ep) \ar[r] & 1 \\
}
$$
where $\Lie{so}(q)$ and $\Lie{g}$ are the Lie algebras of $\SOrth(q)$ and
$R_{k_\ep/k}\PGL_1(\CliffB_0(q))$, respectively.

The Lie algebra $\Lie{so}(q_1)$ of $\Orth(q_1)$ is identified with the
scheme of $3\times 3$ matrices $A$ such that $AQ_1$ is skew-symmetic,
where $Q_1 = \text{diag}(1,-1,1)$.  It is then a consequence of
\eqref{eq:relations} that $I + x\, \Lie{so}(q)(k)$ consists of block matrices of the
form
$$
\begin{pmatrix}
I + x A & 0 \\
x w & 1
\end{pmatrix}
$$
for $w \in \A^3(k)$ and $A \in \Lie{so}(q_1)(k)$.   Since
$$
\begin{pmatrix}
I +  x A & 0 \\
x w & 1
\end{pmatrix}
=
\begin{pmatrix}
I + x A & 0 \\
0 & 1
\end{pmatrix}
\begin{pmatrix}
I  & 0 \\
x w & 1
\end{pmatrix}
=
\begin{pmatrix}
I  & 0 \\
x w & 1
\end{pmatrix}
\begin{pmatrix}
I + x A & 0 \\
0 & 1
\end{pmatrix},
$$
we see that $I+x\,\Lie{so}(q)$ has a direct product decomposition
$\Group{E} \times \bigl(I + x\, \Lie{so}(q_1)\bigr)$.  We claim that
the map $\Lie{h} \to \Lie{g}$ is explicitly given by the product map
$$
\begin{pmatrix}
I + x A & 0 \\
x w & 1
\end{pmatrix}
\mapsto \bigl(I - \ep\beta(x w)\bigr) \bigl(I - \alpha(x
A)\bigr) = I - x(\alpha(A) + \ep \beta(w))
$$
where $\alpha : \Lie{so}(q_1)\to \Lie{sl}_2$ is the Lie algebra
isomorphism 
$$
\begin{pmatrix}
0 &a & -b \\
a  &0 & c  \\
b  &c  & 0  \\
\end{pmatrix}
\mapsto
\frac{1}{2}
\begin{pmatrix}
a & -b+c \\
b+c & -a
\end{pmatrix}
$$
induced from the isomorphism $\PSOrth(q_1) \isom \PGL_2$ and $\beta :
\A^3 \to \Lie{sl}_2$ is the Lie algebra isomorphism
$$
(a~\; b~\; c)
\mapsto
\frac{1}{2}
\begin{pmatrix}
a & -b+c \\
b+c & -a
\end{pmatrix}
$$
as above. Thus $\text{d}c : \Lie{so}(q) \to \Lie{g}$ is an
isomorphism.  
\end{proof}

\begin{remark}
The isomorphism of algebraic groups in the proof of
Theorem~\ref{thm:isom} can be viewed as a degeneration of an
isomorphism of semisimple groups of type ${}^2\Dynkin{A}_1 =
\Dynkin{D}_2$ (on the generic fiber) to an isomorphism of nonreductive
groups whose semisimplification has type $\Dynkin{A}_1 =
\Dynkin{B}_1 = \Dynkin{C}_1$ (on the special fiber).
\end{remark}

\section{Simple degeneration over semi-local rings}
\label{sec:preliminary_results}

The semilocal ring $R$ of a normal scheme at a finite set of points of
codimension 1 is a semilocal Dedekind domain, hence a principal ideal
domain.  Let $R_i$ denote the (finitely many) discrete valuation
overrings of $R$ contained in the fraction field $K$ (the
localizations at the height one prime ideals), $\wh{R}_{i}$ their
completions, and $\wh{K}_{i}$ their fraction fields.  If $\wh{R}$ is
the completion of $R$ at its Jacobson radical $\rad(R)$ and $\wh{K}$
the total ring of fractions, then $\wh{R} \isom \prod_i \wh{R}_i$ and
$\wh{K} \isom \prod_i \wh{K}_i$.  We call an element $\pi \in R$ a
\linedef{parameter} if $\pi = \prod_i \pi_i$ is a product of
parameters $\pi_i$ of $R_i$.

We first recall a well-known result, cf.\
\cite[\S2.3.1]{colliot-thelene:semi-local}.

\begin{lemma}
\label{lem:local_representation}
Let $R$ be a semilocal principal ideal domain and $K$ its field of
fractions.  Let $q$ be a regular quadratic form over $R$ and $u \in
R\mult$ a unit.  If $u$ is represented by $q$ over $K$ then it is
represented by $q$ over $R$.
\end{lemma}

We now provide a generalization of Lemma
\ref{lem:local_representation} to the case of simple degeneration.

\begin{prop}
\label{prop:local_representation}
Let $R$ be a semilocal principal ideal domain with 2 invertible and
$K$ its field of fractions.  Let $q$ be a quadratic form over $R$ with
simple degeneration of multiplicity one and let $u \in R\mult$ be a
unit.  If $u$ is represented by $q$ over $K$ then it is represented
by $q$ over $R$.
\end{prop}

For the proof, we'll first need to generalize, to the degenerate case,
some standard results concerning regular forms.  If $(V,q)$ is a
quadratic form over a ring $R$ and $v \in V$ is such that $q(v) = u
\in R\mult$, then the \linedef{reflection} $r_v : V \to V$ through $v$
given by
$$
r_v(w) = w - u\inv b_q(v,w) \, v
$$
is an isometry over $R$ satisfying $r_v(v) = -v$ and $r_v(w)=w$ if $w
\in v^{\perp}$.

\begin{lemma}
\label{lem:transitive}
Let $R$ be a semilocal ring with 2 invertible.  Let $(V,q)$ be a
quadratic form over $R$ and $u \in R\mult$.  Then $\Orth(V,q)(R)$ acts
transitively on the set of vectors $v \in V$ such that $q(v) = u$.
\end{lemma}
\begin{proof}
Let $v, w \in V$ be such that $q(v)=q(w)=u$.  We first prove the lemma
over any local ring with 2 invertible.  Without loss of generality, we
can assume that $q(v - w) \in R\mult$.  Indeed, $q(v + w) + q(v - w) =
4u \in R\mult$ so that, since $R$ is local, either $q(v + w)$ or $q(v
- w)$ is a unit.  If $q(v-w)$ is not a unit, then $q(v + w)$ is and we
can replace $w$ by $-w$ using the reflection $r_{w}$.  Finally, by a
standard computation, we have $r_{v-w}(v) = w$.  Thus any two vectors
representing $u$ are related by a product of at most two reflection.

For a general semilocal ring, the quotient $R/\rad(R)$ is a product of
fields.  By the above argument, $\ol{v}$ can be transported to
$-\ol{w}$ in each component by a product $\ol{\tau}$ of at most two
reflections.  By the Chinese remainder theorem, we can lift
$\ol{\tau}$ to a product of at most two reflections $\tau$ of $(V,q)$
transporting $v$ to $-w+z$ for some $z \in \rad(R) \tensor_R V$.
Replacing $v$ by $-w+z$, we can assume that $v+w=z \in \rad(R)\tensor_R
V$.  Finally, $q(v+w) + q(v-w)= 4u$ and $q(v+w) \in \rad(R)$, thus
$q(v-w)$ is a unit. As before, $r_{v-w}(v)=w$.
\end{proof}

\begin{cor}
\label{cor:cancel_unit}
Let $R$ be a semilocal ring with 2 invertible.  Then regular forms can be
cancelled, i.e. if $q_1$ and $q_2$ are quadratic forms and $q$ a
regular quadratic form over $R$ with $q_1 \perp q \isom q_2 \perp q$,
then $q_1 \isom q_2$.
\end{cor}
\begin{proof}
Regular quadratic forms over a semilocal ring with 2 invertible are
diagonalizable. Hence we can reduce to the case of rank one form $q =
(R,\qf{u})$ for $u \in R\mult$.  Let $\vp : q_1 \perp (R\, w_1,\qf{u})
\isom q_2 \perp (R\, w_2,\qf{u})$ be an isometry.  By Lemma
\ref{lem:transitive}, there is an isometry $\psi$ of $q_2 \perp (R\,
w_2,\qf{u})$ taking $\vp(w_1)$ to $w_2$, so that $\psi\circ\vp$ takes
$w_1$ to $w_2$.  By taking orthogonal complements, $\vp$ thus induces
an isometry $q_1 \isom q_2$.
\end{proof}

\begin{lemma}
\label{lem:complete_rep}
Let $R$ be a complete discrete valuation ring with 2 invertible and
$K$ its fraction field.  Let $q$ be a quadratic form with simple
degeneration of multiplicity one and let $u \in R\mult$.  If $u$ is
represented by $q$ over $K$ then it is represented by $q$ over $R$.
\end{lemma}
\begin{proof}
For a choice of parameter $\pi$ of $R$, write $(V,q) = (V_1,q_1) \perp
(R\, e,\qf{\pi})$ with $q_1$ regular.  There are two cases, depending
on whether $\ol{q}_1$ is isotropic over the residue field.  First, if
$\ol{q}_1$ is anisotropic, then $q_1$ only takes values with even
valuation.  Let $v \in V_K$ satisfy $q|_K(v)=u$ and write $v = \pi^n
v_1 + a\pi^m e$ with $v_1 \in V_1$ such that $\ol{v}_1 \neq 0$ and $a
\in R\mult$.  Then we have $\pi^{2n} q_1(v_1) + a\pi^{2m+1} = u$.  By
parity considerations, we see that $n=0$ and m$\geq 0$ are forced, and
thus $m \geq 0$, thus $v \in V$ and $u$ is represented by $q$.
Second, $R$ being complete, if $\ol{q}_1$ is isotropic, then
it splits off a hyperbolic plane, so represents $u$. 
\end{proof}

We now recall the theory of elementary hyperbolic isometries initiated
by Eichler \cite[Ch.~1]{eichler:orthogonal_groups} and developed in
the setting of regular quadratic forms over rings by Wall
\cite[\S5]{wall:orthogonal} and Roy \cite[\S5]{roy:cancellation}.  See
also \cite{ojanguren:polynomial_algebras},
\cite{parimala:Dedekind_rings}, \cite{suresh:Eichler}, and
\cite[III~\S2]{baeza:semilocal_rings}.  We will need to develop the
theory for quadratic forms that are not necessarily regular.

Let $R$ be a ring with 2 invertible, $(V,q)$ a quadratic form over
$R$, and $(R^2,h)$ the hyperbolic plane with basis $e,f$.  For $v \in
V$, define $E_v$ and $E_v^*$ in $\Orth(q\perp h)(R)$ by
$$
\begin{minipage}{4cm}
\begin{align*}
E_v(w) & = w + b(v,w)e \\
E_v(e)\; & = e \\
E_v(f)\, & = -v - 2\inv q(v) e + f
\end{align*}
\end{minipage}
\qquad\qquad
\begin{minipage}{4cm}
\begin{align*}
E_v^*(w) & = w + b(v,w)f, \qquad\quad \text{for $w \in V$}\\
E_v^*(e)\; & = -v - 2\inv q(v) f + e \\
E_v^*(f)\, & = f.
\end{align*}
\end{minipage}
$$
Define the group of \linedef{elementary hyperbolic isometries}
$\EOrth(q,h)(R)$ to be the subgroup of $\Orth(q\perp h)(R)$
generated by $E_v$ and $E_v^*$ for $v \in V$.  

For $u \in R\mult$, define $\alpha_u \in \Orth(h)(R)$ by
$$
\alpha_u(e)=ue, \qquad \alpha_u(f) = u\inv f
$$
and $\beta_u \in \Orth(h)(R)$ by
$$
\beta_u(e)=u\inv f, \qquad \beta_u(f) = u e.
$$
Then $\Orth(h)(R) = \{\alpha_u \, : \, u \in R\mult\} \cup \{\beta_u
\, : \, u \in R\mult\}$.  One can verify the following identities:
\begin{align*}
\alpha_u\inv E_v \alpha_u = E_{u\inv v}, & \qquad 
\beta_u\inv E_v \beta_u = E_{v}^*, \\
\alpha_u\inv E_v^* \alpha_u = E_{u\inv v}^*, & \qquad 
\alpha_u\inv E_v^* \alpha_u = E_{v}.
\end{align*}
Thus $\Orth(h)(R)$ normalizes $\EOrth(q, h)(R)$.  

If $R=K$ is a field and $q$ is nondegenerate, then $\EOrth(q,h)(K)$
and $\Orth(h)(K)$ generate $\Orth(q\perp h)(K)$ (see
\cite[ch.~1]{eichler:orthogonal_groups}) so that
\begin{equation}
\label{eq:generate}
\Orth(q\perp h)(K) = \EOrth(q,h)(K) \rtimes \Orth(h)(K).
\end{equation}

\begin{prop}
\label{prop:factorization}
Let $R$ be a semilocal principal ideal domain with 2 invertible and
$K$ its fraction field.  Let $\wh{R}$ be the completion of $R$ at the
radical and $\wh{K}$ its fraction field.  Let $(V,q)$ be a quadratic
form over $R$ that is nondegenerate over $K$.  Then every element $\vp
\in \Orth(q\perp h)(\wh{K})$ is a product $\vp_1 \vp_2$, where $\vp_1
\in \Orth(q \perp h)(K)$ and $\Orth(q\perp h)(\wh{R})$.
\end{prop}
\begin{proof}
We follow portions of the proof in
\cite[Prop.~3.1]{parimala:Dedekind_rings}.  As topological rings
$\wh{R}$ is open in $\wh{K}$, and hence as topological groups
$\Orth(q\perp h)(\wh{R})$ is open in $\Orth(q\perp h)(\wh{K})$.  In
particular, $\Orth(q \perp h)(\wh{R}) \cap \EOrth(q,h)(\wh{K})$ is
open in $\EOrth(q,h)(\wh{K})$.  Since $R$ is dense in $\wh{R}$, $K$ is
dense in $\wh{K}$, $V\tensor_R K$ is dense in $V \tensor_R \wh{K}$,
and hence $\EOrth(q,h)(K)$ is dense in $\EOrth(q,h)(\wh{K})$.

Thus, by topological considerations, every element $\vp'$ of
$\EOrth(q\perp h)(\wh{K})$ is a product $\vp_1' \vp_2'$, where $\vp_1'
\in \EOrth(q,h)(K)$ and $\vp_2' \in \EOrth(q,h)(\wh{K}) \cap
\Orth(q\perp h)(\wh{R})$.  Clearly, every element $\gamma$ of
$\Orth(h)(\wh{K})$ is a product $\gamma_1 \gamma_2$, where $\gamma_1
\in \Orth(h)(K)$ and $\gamma_2 \in \Orth(h)(\wh{R})$.

The form $q \perp h$ is nondegenerate over $\wh{K}$, so by
\eqref{eq:generate}, every $\vp \in \Orth(q\perp h)(\wh{K})$ is a
product $\vp' \gamma$, where $\vp' \in \EOrth(q,h)(\wh{K})$ and
$\gamma \in \Orth(h)(\wh{K})$.  As above, we can write
$$
\vp = \vp' \gamma = \vp_1' \vp_2' \gamma_1 \gamma_2 = \vp_1' \gamma_1
(\gamma_1\inv \vp_2' \gamma_1) \gamma_2.
$$
Since $\EOrth(q,h)(\wh{K})$ is a normal subgroup,
$\gamma_1\inv\vp_2'\gamma_1 \in \EOrth(q,h)(\wh{K})$ and is thus a
product $\psi_1\psi_2$, where $\psi_1 \in \EOrth(q,h)(K)$ and $\psi_2
\in \EOrth(q,h)(\wh{K}) \cap \Orth(q\perp h)(\wh{R})$.  Finally, $\vp$
is a product $(\vp_1'\gamma_1 \psi_1)(\psi_2\gamma_2)$ of the desired
form.
\end{proof}

\begin{proof}[Proof of Proposition~\ref{prop:local_representation}]
Let $\wh{R}$ be the completion of $R$ at the radical and $\wh{K}$ the
total ring of fractions.  As $q|_K$ represents $u$, we have a
splitting $q|_K \isom q_{1} \perp \qf{u}$.  We have that $q|_{\wh{R}}
= \prod_i q|_{\wh{R}_i}$ represents $u$ over $\wh{R} = \prod_i
\wh{R}_i$, by Lemma~\ref{lem:complete_rep}, since $u$ is represented
over $\wh{K} = \prod_i \wh{K}_i$. We thus have a splitting
$q|_{\wh{R}} \isom q_2 \perp \qf{u}$.  By Witt cancellation over
$\wh{K}$, we have an isometry $\vp : q_{1}|_{\wh{K}} \isom
q_{2}|_{\wh{K}}$, which by patching defines a quadratic form
$\tilde{q}$ over $R$ such that $\tilde{q}_{K} \isom q_1$ and
$\tilde{q}|_{\wh{R}} \isom q_2$.

We claim that $q \perp \qf{-u} \isom \tilde{q} \perp h$.  Indeed, as
$h \isom \qf{u,-u}$, we have isometries
$$
\psi^K : (q \perp \qf{-u})_K \isom
(\tilde{q}\perp h)_K, \qquad 
\psi^{\wh{R}} : (q \perp \qf{-u})_{\wh{R}} \isom
(\tilde{q}\perp h)_{\wh{R}}.
$$  
By Proposition~\ref{prop:factorization}, there exists $\theta_1 \in
\Orth(\tilde{q}\perp h)(\wh{R})$ and $\theta_2 \in
\Orth(\tilde{q}\perp h)(K)$ such that $\psi^{\wh{R}} (\psi^{K})\inv =
\theta_1\inv\theta_2$.  The isometries $\theta_1 \psi^{\wh{R}}$ and
$\theta_2\psi^K$ then agree over $\wh{K}$ and so patch to yield an
isometry $\psi : q \perp \qf{-u} \isom \tilde{q} \perp h$.

As $h \isom \qf{u,-u}$, we have $q \perp \qf{-u} \isom \tilde{q}
\perp \qf{u,-u}$.  By Corollary~\ref{cor:cancel_unit}, we can cancel
the regular form $\qf{-u}$, so that $q \isom \tilde{q}\perp \qf{u}$.
Thus $q$ represents $u$ over $R$.
\end{proof}

\begin{lemma}
\label{lem:diag}
Let $R$ be a discrete valuation ring and $(E,q)$ a quadratic form of
rank $n$ over $R$ with simple degeneration.  If $q$ represents $u \in
R\mult$ then it can be diagonalized as $q \isom \qf{u, u_2, \dotsc,
u_{n-1}, \pi}$ for $u_i \in R\mult$ and some parameter $\pi$.
\end{lemma}
\begin{proof}
If $q(v) = u$ for some $v \in E$, then $q$ restricted to the submodule
$Rv \subset E$ is regular, hence $(E,q)$ splits as $(R,\qf{u})\perp
(Rv^{\perp},q|_{Rv^{\perp}})$.  Since $(Rv^{\perp},q|_{Rv^{\perp}})$
has simple degeneration, we are done by induction.  
\end{proof}

\begin{cor}
\label{cor:injectivity}
Let $R$ be a semilocal principal ideal domain with 2 invertible and
$K$ its fraction field.  If quadratic forms $q$ and $q'$ with simple
degeneration and multiplicity one over $R$ are isometric over $K$,
then they are isometric over $R$.
\end{cor}
\begin{proof}
Any quadratic form $q$ with simple degeneration and multiplicity one
has discriminant $\pi \in R/R\mult{}^2$ given by a parameter.  Since
$R\mult/R\mult{}^2 \to K\mult/K\mult{}^2$ is injective, if $q'$ is
another quadratic form with simple degeneration and multiplicity one,
such that $q|_K$ is isomorphic to $q|_K'$, then $q$ and $q'$ have the
same discriminant.  

Over each discrete valuation overring $R_i$ of $R$, we thus have
diagonalizations,
$$
q|_{R_i} \isom \qf{ u_1, \dotsc, u_{r-1}, u_1 \dotsm u_{r-1} \pi_i },
\quad 
q'|_{R'_i} \isom \qf{ u_1', \dotsc, u_{r-1}', u_1' \dotsm
u_{r-1}' \pi_i },
$$
for a suitable parameter $\pi_i$ of $R_i$, where $u_j, u_j' \in
R_i\mult$.  Now, since $q|_K$ and $q'|_K$ are isometric, $q'|_K$
represents $u_1$ over $K$, hence by
Proposition~\ref{prop:local_representation}, $q'|_{R_i}$ represents
$u_1$ over $R_i$. Hence by Lemma~\ref{lem:diag}, we have a further
diagonalization
$$
q'|_{R_i} \isom \qf{ u_1, u_2', \dotsc, u_{r-1}', u_1 u_2' \dotsm u_{r-1}'\pi_i }
$$
with possibly different units $u_i'$.  By cancellation over $K$, we
have $\qf{ u_2, \dotsc, u_{r-1}, u_1 \dotsm u_{r-1} \pi_i }|_K \isom
\qf{ u_2', \dotsc, u_{r-1}', u_1' \dotsm u_{r-1}' \pi_i }|_K$.  By an
induction hypothese over the rank of $q$, we have that $\qf{ u_2,
\dotsc, u_{r-1}, u_1 \dotsm u_{r-1} \pi_i } \isom \qf{ u_2', \dotsc,
u_{r-1}', u_1' \dotsm u_{r-1}' \pi_i }$ over $R$.  By induction, we
have the result over each $R_i$.   

Thus $q|_{\wh{R}} \isom
q'|_{\wh{R}}$ over $\wh{R} = \prod_i \wh{R}_i$.  Consider the induced
isometry $\psi^{\wh{R}} : (q\perp h)|_{\wh{R}} \isom (q'\perp
h)|_{\wh{R}}$ as well as the isometry $\psi^K : (q\perp h)|_{K} \isom
(q'\perp h)|_{K}$ induced from the given one.  By Proposition
\ref{prop:factorization}, there exists $\theta^{\wh{R}} \in
\Orth(q\perp h)(\wh{R})$ and $\theta^K \in \Orth(q\perp h)(K)$ such
that $\psi^{\wh{R}}{}\inv \psi^{K}{} = \theta^{\wh{R}}\theta^K{}\inv$.
The isometries $\psi^{\wh{R}}\theta^{\wh{R}}$ and $\psi^K \theta^K$
then agree over $\wh{K}$ and so patch to yield an isometry $\psi : q
\perp h \isom q' \perp h$ over $R$.  By Corollary
\ref{cor:cancel_unit}, we then have an isometry $q \isom q'$.
\end{proof}

\begin{remark}
Let $R$ be a semilocal principal ideal domain with 2 invertible,
closed fiber $D$, and fraction field $K$. Let $\QF^D(R)$ be the set of
isometry classes of quadratic forms on $R$ with simple degeneration of
multiplicity one along $D$.  Corollary~\ref{cor:injectivity} says that
$\QF^D(R) \to \QF(K)$ is injective, which can be viewed as an analogue
of the Grothendieck--Serre conjecture for the (nonreductive)
orthogonal group of a quadratic form with simple degeneration of
multiplicity one over a discrete valuation ring.
\end{remark}

\begin{cor}
\label{prop:injectivity_sim}
Let $R$ be a complete discrete valuation ring with 2 invertible and
$K$ its fraction field.  If quadratic forms $q$ and $q'$ of even rank
$n=2m \geq 4$ with simple degeneration and multiplicity one over $R$
are similar over $K$, then they are similar.
\end{cor}
\begin{proof}
Let $\psi : q|_K \similar q'|_K$ be a similarity with factor $\lambda
= u \pi^e$ where $u \in R\mult$ and $\pi$ is a parameter whose square
class we can assume is the discriminant of $q$ and $q'$.  If $e$ is
even, then $\pi^{e/2}\psi : q|_K \similar q'|_K$ has factor $u$, so
defines an isometry $q|_K \isom u q'|_K$.  Hence by
Corollary~\ref{cor:injectivity}, there is an isometry $q \isom u q'$,
hence a similarity $q \similar q'$.  If $e$ is odd, then
$\pi^{(e-1)/2}\psi^{K}$ defines an isometry $q|_K \isom u\pi q'|_K$.
Writing $q \isom q_1 \perp \qf{a\pi}$ and $q' = q_1' \perp \qf{b\pi}$
for regular quadratic forms $q_1$ and $q_1'$ over $R$ and $a,b \in
R\mult$, then $\pi q|_K \isom u\pi q'|_K \isom u\pi q_1' \perp
\qf{bu}$.  Comparing first residues, we have that $\ol{q}_1$ and
$\qf{\ol{bu}}$ are equal in $W(k)$, where $k$ is the residue field of
$R$.  
Since $R$ is complete, $q_1$ splits off the requisite number of
hyperbolic planes, and so $q_1 \isom h^{m-1} \perp \qf{(-1)^{m-1}a}$.
Now note that $(-1)^{m-1}\pi$ is a similarity factor of the form
$q|_K$.  Finally, we have $(-1)^{m-1}\pi q|_K \isom q|_K \isom u\pi
q'|_K$, so that $q|_K \isom (-1)^{m-1}u q'|_K$.  Thus by
Corollary~\ref{cor:injectivity}, $q \isom (-1)^mu q'$ over $R$, so
that there is a similarity $q \similar q'$ over $R$.
\end{proof}

We need the following relative version of Theorem~\ref{thm:isom}.

\begin{prop}
\label{prop:rank_4_sim}
Let $R$ be a semilocal principal ideal domain with 2 invertible and
$K$ its fraction field.  Let $q$ and $q'$ be quadratic forms of rank 4
over $R$ with simple degeneration and multiplicity one.
Given any $R$-algebra isomorphism $\vp : \CliffB_0(q) \isom
\CliffB_0(q')$
there exists a similarity $\psi : q \similar q'$ such that
$\CliffB_0(\psi)=\vp$.
\end{prop}
\begin{proof}
By Theorem~\ref{thm:etale}, there exists a similarity $\psi^K : q
\similar q'$ such that $\CliffB_0(\psi^K) = \vp|_K$.  Thus over
$\wh{R} = \prod_i \wh{R}_i$, Corollary~\ref{prop:injectivity_sim}
applied to each component provides a similarity $\rho : q|_{\wh{R}}
\similar q'|_{\wh{R}}$.  Now $\CliffB_0(\rho)\inv \circ \vp :
\CliffB_0(q)|_{\wh{R}} \isom \CliffB_0(q)|_{\wh{R}}$ is a
${\wh{R}}$-algebra isomorphism, hence by Theorem~\ref{thm:isom}, is
equal to
$\CliffB_0(\sigma)$ for some similarity $\sigma : q|_{\wh{R}} \similar
q|_{\wh{R}}$.  Then $\psi^{\wh{R}} = \rho \circ \sigma : q|_{\wh{R}}
\similar q'|_{\wh{R}}$ satisfies $\CliffB_0(\psi^{\wh{R}}) =
\vp|_{\wh{R}}$.

Let $\lambda \in K\mult$ and $u \in \wh{R}\mult$ be the factor of
$\psi^{K}$ and $\psi^{\wh{R}}$, respectively.  Then
$\psi^K|_{\wh{K}}\inv \circ \psi^{\wh{R}}|_{\wh{K}} : q|_{\wh{K}}
\similar q|_{\wh{K}}$ has factor $u\inv\lambda \in \wh{K}\mult$.  But
since $\CliffB_0(\psi^K|_{\wh{K}}\inv \circ \psi^{\wh{R}}|_{\wh{K}}) =
\id$, we have that $\psi^K|_{\wh{K}}\inv \circ
\psi^{\wh{R}}|_{\wh{K}}$ is given by multiplication by $\mu \in
\wh{K}\mult$.  In particular, $u\inv \lambda = \mu^2$ and thus the
valuation of $\lambda \in K\mult$ is even in every $R_i$.  Thus
$\lambda = v \varpi^2$ with $v \in R\mult$ and so $\varpi \psi$
defines an isometry $q|_K \isom v q'|_K$.  By Corollary
\ref{cor:injectivity}, there's an isometry $\alpha : q \isom v q'$,
i.e., a similarity $\alpha : q \similar q'$.  As before,
$\CliffB_0(\alpha)\inv \circ \vp : \CliffB_0(q) \isom \CliffB_0(q)$ is
a $R$-algebra isomorphism, hence by Theorem~\ref{thm:isom}, is equal
to $\CliffB_0(\beta)$ for some similarity $\beta : q \similar q$.
Then we can define a similarity $\psi = \alpha \circ \beta : q
\similar q'$ over $R$, which 
satisfies $\CliffB_0(\psi) = \vp$.
\end{proof}

Finally, we need the following generalization of
\cite[Prop.~2.3]{colliot-thelene_sansuc:fibres_quadratiques} to the
setting of quadratic forms with simple degeneration.

\begin{prop}
\label{prop:extend}
Let $S$ be the spectrum of a regular local ring $(R,\mm)$ of dimension
$\geq 2$ with 2 invertible and $D \subset S$ a regular divisor.  Let
$(V,q)$ be a quadratic form over $S$ such that $(V,q)|_{S\bslash
\{\mm\} }$ has simple degeneration of multiplicity one along $D\bslash
\{\mm\}$.  Then $(V,q)$ has simple degeneration along $D$ of
multiplicity one.
\end{prop}
\begin{proof}
First note that the discriminant of $(V,q)$ (hence the subscheme $D$)
is represented by a regular element $\pi \in \mm \bslash \mm^2$.  Now
assume, to get a contradiction, that the radical of
$(V,q)_{\kappa(\mm)}$, where $\kappa(\mm)$ is the residue field at
$\mm$, has dimension $r>1$ and let $\ol{e}_1, \dotsc, \ol{e}_r$ be a
$\kappa(\mm)$-basis of the radical.  Lifting to unimodular elements $e_1,
\dotsc, e_r$ of $V$, we can complete to a basis $e_1, \dotsc, e_n$.
Since $b_q(e_i,e_j) \in \mm$ for all $1 \leq i \leq r$ and $1 \leq j
\leq n$, inspecting the Gram matrix $M_q$ of $b_q$ with respect to
this basis, we find that $\det M_q \in \mm^r$, contradicting the
description of the discriminant above.  Thus the radical of $(V,q)$
has rank 1 at $\mm$ and $(V,q)$ has simple degeneration along $D$.
Similarly, $(V,q)$ also has multiplicity one at $\mm$, hence on $S$ by
hypothesis.
\end{proof}

\begin{cor}
\label{cor:extending_simple_deg}
Let $S$ be a regular integral scheme of dimension $\leq 2$ with 2
invertible and $D$ a regular divisor.  Let $(\EE,q,\LL)$ be a
quadratic form over $S$ that is regular over every codimension 1 point
of $S \bslash D$ and has simple degeneration of multiplicity one over
every codimension one point of $D$.  Then over $S$, $q$ has simple degeneration
along $D$ of
multiplicity one.
\end{cor}
\begin{proof}
Let $U=S\bslash D$.  The quadratic form $q|_U$ is regular except
possibly at finitely many closed points.  But regular quadratic forms
over the complement of finitely many closed points of a regular
surface extend uniquely by
\cite[Prop.~2.3]{colliot-thelene_sansuc:fibres_quadratiques}.  Hence
$q|_U$ is regular.  The restriction $q|_{D}$ has simple degeneration
at the generic point of ${D}$, hence along the complement of finitely
many closed points of ${D}$.  At each of these closed points, $q$ has
simple degeneration by Proposition~\ref{prop:extend}.  Thus $q$ has
simple degeneration along ${D}$.
\end{proof}

\section{Gluing tensors}
\label{sec:gluing}

In this section, we reproduce some results on gluing (or patching)
tensor structures on vector bundles communicated to us by M.\
Ojanguren and inspired by Colliot-Th\'el\`ene--Sansuc
\cite[\S2,~\S6]{colliot-thelene_sansuc:fibres_quadratiques}.  As usual,
any scheme $S$ is assumed to be noetherian.

\begin{lemma}
\label{lem:top}
Let $S$ be a scheme of dimension $n$, $U \subset S$ a dense open
subset, $x \in S \bslash U$ a point of codimension 1 of $S$, $V
\subset S$ an open neighborhood of $x$, and $W \subset U \cap V$ a
dense open subset of $S$.  Then there exists an open neighborhood $V'$
of $x$ such that $V' \cap U \subset W$.
\end{lemma}
\begin{proof}
The closed set $Z = S \bslash W$ is of dimension $n-1$ and contains
$x$, hence has a decomposition $Z = Z_1 \cup Z_2$, where
$$
Z_1 = \ol{\{x\}} \cup \ol{\{x_1\}} \cup \dotsm \cup \ol{\{x_r\}},
\qquad 
Z_2 = \ol{\{y_1\}} \cup \dotsm \cup \ol{\{y_s\}}
$$
into closed irreducible sets with distinct generic points, where
$x,x^1,\dotsc,x^r \notin U$ and $y_1,\dotsc,y_s \in U$. Let $F =
\{y_1\} \cup \dotsm \cup \{ y_s\}$ and $V'' = S \bslash F$.  Note that
$W \subset V''$.  Since $Z_1 \cap U = \varnothing$, we have that $U
\cap F = U \cap Z$.  This shows that $V'' \cap U = U \bslash (U \cap
Z) = U \cap (X \bslash Z) = U \cap W = W$.  Thus we can take $V' = V
\cap V''$.
\end{proof}

Let $\VV$ be a locally free $\OO_S$-module (of finite rank).  A
\linedef{tensorial construction} $t(\VV)$ in $\VV$ is any locally free
$\OO_S$-module that is a tensor product of modules $\exterior^j(\VV)$,
$\exterior^j(\VV\dual)$, $S^j(\VV)$, or $S^j(\VV\dual)$.  Let $\LL$ be
a line bundle on $S$.  An \linedef{$\LL$-valued tensor} $(\VV,q,\LL)$
of \linedef{type} $t(\VV)$ on $S$ is a global section $q \in
\Gamma(S,t(\VV)\tensor\LL)$ for some tensorial construction $t(\VV)$
in $\VV$.  For example, an $\LL$-valued quadratic form is an
$\LL$-valued tensor of type $t(\VV) = S^2(\VV\dual)$; an
$\OO_S$-algebra structure on $\VV$ is an $\OO_S$-valued tensor of type
$t(\VV) = \VV\dual \tensor \VV\dual \tensor \VV$.  If $U \subset S$ is
an open set, denote by $(\VV,q,\LL)|_U = (\VV|_U,q|_U,\LL|_U)$ the
restricted tensor over $U$.  If $D \subset S$ is a closed subscheme,
let $\OO_{S,D}$ denote the semilocal ring at the generic points of $D$
and $(\VV,q,\LL)|_D = (\VV,q,\LL)\tensor_{\OO_S}\OO_{S,D}$ the
associated tensor over $\OO_{S,D}$.  If $S$ is integral and $K$ its
function field, we write $(\VV,q,\LL)|_K$ for the stalk at the generic
point.  

A \linedef{similarity} between line bundle-valued tensors
$(\VV,q,\LL)$ and $(\VV',q',\LL')$ consists of a pair $(\vp,\lambda)$
where $\vp : \VV \isom \VV'$ and $\lambda : \LL \isom \LL'$ are
$\OO_S$-module isomorphisms such that $t(\vp)\tensor\lambda :
t(\VV)\tensor\LL \isom t(\VV')\tensor\LL'$ takes $q$ to $q'$.  A
similarity is an \linedef{isomorphism} if $\LL=\LL'$ and
$\lambda=\id$.

\begin{prop}
\label{prop:glue}
Let $S$ be an integral scheme, $K$ its function field, $U \subset S$ a
dense open subscheme, and $D \subset S \bslash U$ a closed subscheme
of codimension 1.  Let $(\VV^U,q^U,\LL^U)$ be a tensor over $U$,
$(\VV^D,q^D,\LL^D)$ a tensor over $\OO_{S,D}$, and $(\vp,\lambda) :
(\VV^U,q^U,\LL^U)|_K \similar (\VV^D,q^D,\LL^D)|_K$ a similarity of
tensors over $K$.  Then there exists a dense open set $U' \subset S$
containing $U$ and the generic points of $D$ and a tensor
$(\VV^{U'},q^{U'},\LL^{U'})$ over $U'$ together with similarities
$(\VV^U,q^U,\LL^U) \isom (\VV^{U'},q^{U'},\LL^{U'})|_U$ and
$(\VV^D,q^D,\LL^D) \isom (\VV^{U'},q^{U'},\LL^{U'})|_D$.  A
corresponding statement holds for isomorphisms of tensors.
\end{prop}
\begin{proof}
By induction on the number of irreducible components of $D$, gluing
over one at a time, we can assume that $D$ is irreducible.  Choose an
extension $(\VV^V,q^V,\LL^V)$ of $(\VV^D,q^D,\LL^D)$ to some open
neighborhood $V$ of $D$ in $S$.  Since $(\VV^V,q^V,\LL^V)|_K
\similar (\VV^U,q^U,\LL^U)|_K$, there exists an open subscheme $W
\subset U \cap V$ over which $(\VV^V,q^V,\LL^V)|_W \similar
(\VV^U,q^U,\LL^U)|_W$.  By Lemma~\ref{lem:top}, there exists an open
neighborhood $V' \subset S$ of $D$ such that $V' \cap U \subset W$.
We can glue $(\VV^U,q^U,\LL^U)$ and $(\VV^{V'},q^{V'},\LL^{V'})$ over $U \cap
V'$ to get a tensor $(\VV^{U'},q^{U'},\LL^{U'})$ over $U'$ extending
$(\VV^U,q^U,\LL^U)$, where $U' = U \cup V'$.  But $U'$ contains the
generic points of $D$
and we are done.
\end{proof}

For an open subscheme $U \subset S$, a closed subscheme $D \subset S
\bslash U$ of codimension 1, a \linedef{similarity gluing datum}
(resp.\ \linedef{gluing datum}) is a triple
$((\VV^U,q^U,\LL^U),(\VV^D,q^D,\LL^D),\vp)$ consisting of a tensor
over $U$, a tensor over $\OO_{S,D}$, and a similarity (resp.\ an
isomorphism) of tensors $(\vp,\lambda) : (\VV^U,q^U,\LL^U)|_K \similar
(\VV^D,q^D,\LL^D)|_K$ over $K$.  There is an evident notion of
isomorphism between two (similarity) gluing data.  Two isomorphic
gluing data yield, by Proposition~\ref{prop:glue}, tensors
$(\VV^{U'},q^{U'},\LL^{U'})$ and $(\VV^{U''},q^{U''},\LL^{U''})$ over
open dense subsets $U', U'' \subset S$ containing $U$ and the generic
points of $D$ such that there
is an open dense refinement $U''' \subset U' \cap U''$ over which we
have $(\VV^{U'},q^{U'},\LL^{U'})|_{U'''} \similar
(\VV^{U''},q^{U''},\LL^{U''})|_{U'''}$.

Together with results of
\cite{colliot-thelene_sansuc:fibres_quadratiques}, we get a well-known
result---\linedef{purity for division algebras} over surfaces---which
we state in a precise way, due to Ojanguren, that is conducive to our
usage.  If $K$ is the function field of a regular scheme $S$, we say
that $\beta \in \Br(K)$ is \linedef{unramified} if its in the image of
the injection $\Br(S) \to \Br(K)$.

\begin{theorem}
\label{thm:purity}
Let $S$ be a regular integral scheme of dimension $\leq 2$, $K$ its
function field, $D \subset S$ a closed subscheme of codimension 1, and
$U = S \bslash D$.
\begin{enumerate}
\item If $\AA^U$ is an Azumaya $\OO_U$-algebra such that $\AA^U|_K$ is
unramified along $D$ then there exists an Azumaya $\OO_S$-algebra
$\AA$ such that $\AA|_U \isom \AA^U$.

\item If a central simple $K$-algebra $A$ has Brauer class unramified
over $S$, then there exists an Azumaya $\OO_S$-algebra $\AA$ such that
$\AA|_K \isom A$.
\end{enumerate}
\end{theorem}
\begin{proof}
For \textit{a)}, since $\AA^U|_K$ is unramified along $D$, there
exists an Azumaya $\OO_{S,D}$-algebra $\BB^D$ with $\BB^D|_K$ Brauer
equivalent to $A$.

We argue that we can choose $\BB^D$ such that $\BB^D|_K = A$.  Indeed,
writing $\BB^D|_K = M_m(\Delta)$ for a division $K$-algebra $\Delta$
and choosing a maximal $\OO_{S,D}$-order $\DD^D$ of $\Delta$, then
$M_m(\DD^D)$ is a maximal order of $\BB^D|_K$.  Any two maximal orders
are isomorphic by \cite[Prop.~3.5]{auslander_goldman:maximal_orders},
hence $M_m(\DD^D) \isom \BB^D$.  In particular, $\DD^D$ is an Azumaya
$\OO_{S,D}$-algebra.  Finally writing $A = M_n(\Delta)$, then
$M_n(\DD^D)$ is an Azumaya $\OO_{S,D}$-algebra and is our new choice
for $\BB^D$.

Applying Proposition~\ref{prop:glue} to $\AA^U$ and $\BB^D$, we get an
Azumaya $\OO_{U'}$-algebra $\AA^{U'}$ extending $\AA^U$, where $U'$
contains all points of $S$ of codimension 1.  Finally, by
\cite[Thm.~6.13]{colliot-thelene_sansuc:fibres_quadratiques} applied
to the group $\PGL_n$ (where $n$ is the degree of $A$), $\AA^{U'}$
extends to an Azumaya $\OO_S$-algebra $\AA$ such that
$\AA|_U=\AA^U$.

For \textit{b)}, the $K$-algebra $A$ extends, over some open subscheme
$U \subset S$, to an Azumaya $\OO_U$-algebra $\AA^U$.  If $U$ contains
all codimension 1 points, then we apply
\cite[Thm.~6.13]{colliot-thelene_sansuc:fibres_quadratiques} as above.
Otherwise, $D = S \bslash U$ has codimension 1 and we apply part (1).
\end{proof}

Finally, we note that isomorphic Azumaya algebra gluing data on a
regular integral scheme $S$ of dimension $\leq 2$ yield, by
\cite[Thm.~6.13]{colliot-thelene_sansuc:fibres_quadratiques},
isomorphic Azumaya algebras on $S$.

\section{The norm form $\Norm_{T/S}$ for ramified covers}
\label{sec:cores}

Let $S$ be a regular integral scheme, $D \subset S$ a regular divisor,
and $f : T \to S$ a ramified cover of degree 2 branched along $D$.
Then $T$ is a regular integral scheme.  Let $L/K$ be the corresponding
quadratic extension of function fields.  Let $U = S \bslash D$, and
for $E=f\inv(D)$, let $V = T \bslash E$. Then $f|_V : V \to U$ is
\'etale of degree 2. Let $\iota$ be the nontrivial Galois automorphism
of $T/S$.  

The following lemma is not strictly used in our construction but we
need it for the applicatications in \S\ref{sec:failure}.  

\begin{lemma}
\label{lem:cores_extends}
Let $S$ be a regular integral scheme and $f : T \to S$ a finite flat cover of
prime degree $\ell$ with regular branch divisor $D \subset S$ on which
$\ell$ is invertible.  Let $L/K$ be the corresponding extension of
function fields.
\begin{enumerate}
\item The corestriction map $\Norm_{L/K} : \Br(L) \to \Br(K)$
restricts to a well-defined map $\Norm_{T/S} : \Br(T) \to \Br(S)$.

\item If $S$ has dimension $\leq 2$ and $\BB$ is an Azumaya
$\OO_T$-algebra of degree $d$ represeting $\beta \in \Br(T)$ then
there exists an Azumaya $\OO_S$-algebra of degree $\ell d$ representing
$\Norm_{T/S}(\beta)$ whose restriction to $U$ coincides with the
classical \'etale norm algebra $\Norm_{V/U}\BB|_V$.
\end{enumerate}
\end{lemma}
\begin{proof}
The hypotheses imply that $T$ is regular integral and so by
\cite{auslander_goldman:brauer_group_commutative_ring}, we can
consider $\Br(S) \subset \Br(K)$ and $\Br(T) \subset \Br(L)$.  Let
$\BB$ be an Azumaya $\OO_T$-algebra of degree $d$ representing $\beta
\in \Br(T)$.  As $V/U$ is \'etale, the classical norm algebra
$\Norm_{V/U}(\BB|_V)$ is an Azumaya $\OO_U$-algebra of degree $nd$
representing the class of $\Norm_{L/K}(\beta) \in \Br(K)$.
In particular, $\Norm_{L/K}(\beta)$ is unramified at every point (of
codimension 1) in $U$.  As $D$ is regular, it is a disjoint union of
irreducible divisors and let $D'$ be one such irreducible component.
If $E' = f\pullback D'$, then $\OO_{T,E'}$ is totally ramified over
$\OO_{S,D'}$ (since it is ramified of prime degree).  In particular,
$E' \subset T$ is an irreducible component of $E = f\pullback D$.  The
commutative diagram
$$
\xymatrix@C=36pt@R=18pt{
\Br(\OO_{T,E'}) \ar@{..>}[d] \ar[r] & \Br(L) \ar[d]^{\Norm_{L/K}}\ar[r]^(.43){\residue} & \Br(\kappa(E')) \ar@{=}[d]\\
\Br(\OO_{S,D'}) \ar[r] & \Br(K) \ar[r]^(.43){\residue} & \Br(\kappa(D'))
}
$$
of residue homomorphisms implies, since $\beta$ is unramified along
$E'$, that $\Norm_{L/K}(\beta)$ is unramified along $D'$.  
Thus $\Norm_{L/K}(\beta)$ is an unramified class in $\Br(K)$, hence is
contained in $\Br(S)$ by purity for the Brauer group (cf.\
\cite[Cor.~1.10]{grothendieck:Brauer_II}).  This proves \textit{a)}.

By Theorem~\ref{thm:purity}, $\Norm_{V/U}(\BB|_V)$ extends (since by \textit{a)},
it is unramified along $D$) to an Azumaya $\OO_S$-algebra of degree
$\ell d$, whose generic fiber is $\Norm_{L/K}(\beta)$.
\end{proof}

Suppose that $S$ has dimension $\leq 2$.  We are interested in finding
a good extension of $\Norm_{V/U}(\BB|_V)$ to $S$.  We note that if
$\BB$ has an involution of the first kind $\tau$, then the
corestriction involution $\Norm_{V/U}(\tau|_V)$, given by the
restriction of $\iota\pushforward\tau|_V\tensor\tau|_V$ to
$\Norm_{V/U}(\AA|_V)$, is of orthogonal type.  If $\Norm_{V/U}(\BB|_V)
\isom \EEnd(\EE^U)$ is split, then $\Norm_{V/U}(\tau|_V)$ is adjoint
to a regular line bundle-valued quadratic form $(\EE^U,q^U,\LL^U)$ on
$U$ unique up to projective similarity.
The main result of this section is that this extends to a line
bundle-valued quadratic form $(\EE,q,\LL)$ on $S$ with simple
degeneration along a regular divisor $D$ satisfying
$\CliffB_0(\EE,q,\LL) \isom \BB$.

\begin{theorem}
\label{thm:existence}
Let $S$ be a regular integral scheme of dimension $\leq 2$ with 2
invertible and $f : T \to S$ a finite flat cover of degree 2 with regular
branch divisor $D$.  Let $\BB$ be an Azumaya quaternion
$\OO_T$-algebra with standard involution $\tau$.  Suppose that
$\Norm_{V/U}(\BB|_V)$ is split and $\Norm_{V/U}(\tau|_V)$ is adjoint to a
regular line bundle-valued quadratic form $(\EE^U,q^U,\LL^U)$ on $U$.
There exists a line bundle-valued quadratic form $(\EE,q,\LL)$ on $S$
with simple degeneration along $D$ with multiplicity one, which
restricts to $(\EE^U,q^U,\LL^U)$ on $U$ and such that
$\CliffB_0(\EE,q,\LL)\isom\BB$.
\end{theorem}

First we need the following lemma.  Let $S$ be a normal integral
scheme, $K$ its function field, $D \subset S$ a regular divisor, and
$\OO_{S,D}$ the the semilocal ring at the generic points of $D$.

\begin{lemma}
\label{lem:discriminant}
Let $S$ be a normal integral scheme with $2$ invertible, $T \to S$ a
finite flat cover of degree 2 with regular branch divisor $D \subset S$, and
$L/K$ the corresponding extension of function fields.  Under the
restriction map $\Het^1(U,\Z/2\Z) \to \Het^1(K,\Z/2\Z) =
K\mult/K\mult{}^2$, the class of the \'etale quadratic extension
$[V/U]$ maps to a square class
represented by a parameter $\pi \in K\mult$ of the semilocal ring
$\OO_{S,D}$.
\end{lemma}
\begin{proof}
Consider any $\pi \in K\mult$ with $L = K(\sqrt{\pi})$.  For any
irreducible component $D'$ of $D$, if $v_{D'}(\pi)$ is even, then we
can modify $\pi$ up to squares in $K$ so that $v_{D'}(\pi)=0$. But
then $T/S$ would be \'etale at the generic point of $D'$, which is
impossible.  Hence, $v_{D'}(\pi)$ is odd for every irreducible
component $D'$ of $D$.  Since $\OO_{S,D}$ is a principal ideal domain,
we can modify $\pi$ up to squares in $K$ so that $v_{D'}(\pi)=1$ for
every component $D'$ of $D$.  Under $\Het^1(U,\Z/2\Z) \to
\Het^1(K,\Z/2\Z) = K\mult/K\mult{}^2$, the class of $[V/U]$ is mapped
to the class of $[L/K]$, which corresponds via Kummer theory to the
square class $(\pi)$.
\end{proof}

\begin{proof}[Proof of Theorem~\ref{thm:existence}]
If $D = \cup D_i$ is the irreducible decomposition of $D$ and $\pi_i$
is a parameter of $\OO_{S,D_i}$, then $\pi=\prod_i \pi_i$ is a
parameter of $\OO_{S,D}$. Choose a regular quadratic form
$(\EE^U,q^U,\LL^U)$ on $U$ adjoint to $\Norm_{V/S}(\sigma|_V)$.  Since
$\OO_{S,D}$ is a principal ideal domain, modifying by squares over
$K$, the form $q^U|_K$ has a diagonalization $\qf{a_1,a_2,a_3,a_4}$,
where $a_i \in \OO_{S,D}$ are squarefree.  By
Lemma~\ref{lem:discriminant}, we can choose $\pi \in K\mult$ so that
$[V/U] \in \Het^1(U,\Z/2\Z)$ maps to the square class $(\pi)$.  By
Theorem~\ref{thm:etale}, the class $[V/U]$ maps to the discriminant of
$q^U|_K$.  Since $\OO_{S,D}$ is a principal ideal domain, $a_1\dotsm
a_4 = \mu^2 \pi$, for some $\mu \in \OO_{S,D}$.  If $\pi_i$ divides
$\mu$, then $\pi_i$ divides exactly 3 of $a_1, a_2,a_3,a_4$, so that clearing
squares from the entries of $\mu\!\qf{a_1,a_2,a_3,a_4}$ yields a form
$\qf{a_1',a_2',a_3',a_4'}$ over $\OO_{S,D}$ with simple degeneration
along $D$, which over $K$, is isometric to $\mu q^U|_K$.
Define
$$(\EE^D,q^D,\LL^D) = 
(\OO_{S,D}^4,\qf{a_1',a_2',a_3',a_4'},\OO_{S,D}).
$$ 
By definition, the identity map is a similarity $q^U|_K \similar
q^{D}|_K$ with similarity factor $\mu$ (up to $K\mult{}^2$).
Our aim is to find a good similarity enabling a gluing to a quadratic
form over $S$ with simple degeneration along $D$ and the correct even
Clifford algebra.

First note that by the classical theory of
${}^2\Dynkin{A}_1=\Dynkin{D}_2$ over $V/U$ (cf.\ Theorem
\ref{thm:etale}), we can choose an $\OO_V$-algebra isomorphism $\vp^U
: \CliffB_0(\EE^U,q^U,\LL^U) \to \BB|_V$.  Second, we can pick an
$\OO_{T,E}$-algebra isomorphism $\vp^{D} : \CliffB_0(q^{D}) \to
\BB|_{E}$, where $E = f\inv D$.  Indeed, by the classical theory of
${}^2\Dynkin{A}_1=\Dynkin{D}_2$ over $L/K$ (cf.\
Theorem~\ref{thm:etale}), the central simple algebras
$\CliffB_0(q^{D})|_L$ and $\BB|_{L}$ are isomorphic over $L$, hence
they are isomorphic over the semilocal principal ideal domain
$\OO_{T,E}$.  Now consider the $L$-algebra isomorphism $\vp^L =
(\vp^U|_L)\inv \circ \vp^{D}|_L : \CliffB_0(q^{D})|_L \to
\CliffB_0(q^U)|_L$.  Again by the classical theory of
${}^2\Dynkin{A}_1=\Dynkin{D}_2$ over $L/K$ (cf.\ Theorem
\ref{thm:etale}), this is induced by a similarity $\psi^K : q^D|_K \to
q^U|_K$, unique up to multiplication by scalars.  By
Proposition~\ref{prop:glue}, the quadratic forms $(\EE^U,q^U,\LL^U)$
and $(\EE^D,q^{D},\LL^D)$ glue, via the similarity $\psi^K$, to a
quadratic form $(\EE^{U'},q^{U'},\LL^{U'})$ on a dense open subscheme
$U' \subset S$ containing $U$ and the generic points of $D$, hence all
points of codimension 1. 
By \cite[Prop.~2.3]{colliot-thelene_sansuc:fibres_quadratiques}, the
quadratic form $(\EE^{U'},q^{U'},\LL^{U'})$ 
extends uniquely to a quadratic form $(\EE,q,\LL)$ on $S$ since the
underlying vector bundle $\EE^{U'}$ extends to a vector bundle $\EE$
on $S$ (because $S$ is a regular integral schemes of dimension $\leq
2$).  By Corollary~\ref{cor:extending_simple_deg}, this extension has
simple degeneration along $D$.

Finally, we argue that $\CliffB_0(q) \isom \BB$.  We know that
$q|_U=q^U$ and $q|_{D} = q^{D}$ and we have algebra isomorphisms
$\vp^U : \CliffB_0(q)|_U \isom \BB|_U$ and
$\vp^{D} : \CliffB_0(q)|_{D} \isom \BB|_{D}$
such that $\vp^L = (\vp^U|_L)\inv \circ \vp^{D}|_L$.  Hence the gluing
data $(\CliffB_0(q)|_U,\CliffB_0(q)|_{D},\vp^L)$ is isomorphic to the
gluing data $(\BB|_U, \BB|_{D}, \id)$.  Thus $\CliffB_0(q)$ and $\BB$
are isomorphic over an open subset $U' \subset S$ containing all
codimension 1 points of $S$.  Hence by
\cite[Thm.~6.13]{colliot-thelene_sansuc:fibres_quadratiques}, these
Azumaya algebras are isomorphic over $S$.
\end{proof}

Finally, we give the proof of our main result.

\begin{proof}[Proof of Theorem~\ref{thm:main}]  
Theorem~\ref{thm:existence} implies that $\CliffB_0 : \Quadrics_2(T/S)
\to \Az_2(T/S)$ is surjective.  To prove the injectivity, let
$(\EE_1,q_1,\LL_1)$ and $(\EE_2,q_2,\LL_2)$ be line bundle-valued
quadratic forms of rank 4 on $S$ with simple degeneration along $D$ of
multiplicity one such that there is an $\OO_T$-algebra isomorphism $\vp
: \CliffB_0(q_1) \isom \CliffB_0(q_2)$.  By the classical theory of
${}^2\Dynkin{A}_1=\Dynkin{D}_2$ over $V/U$ (cf.\ Theorem
\ref{thm:etale}), we know that $\vp|_U : \CliffB_0(q_1)|_U \isom
\CliffB_0(q_2)|_U$ is induced by a similarity transformation $\psi^U :
q_{1}|_U \similar q_{2}|_U \tensor \NN^U$, for some line bundle
$\NN^U$ on $U$, which we can assume is the restriction of a line
bundle $\NN$ on $S$. Replacing
$(\EE_2,q_2,\LL_2)$ by $(\EE_2\tensor \NN,q_2 \tensor
\qf{1},\LL\tensor\NN^{\tensor 2})$, which is in the same
projective similarity class, we can assume that $\psi^U : q_{1}|_U
\similar q_2|_U$.
In particular, $\LL_1|_U\isom \LL_2|_U$ so that we have $\LL_1 \isom
\LL_2 \tensor \MM$ for some $\MM|_U \isom \OO_U$ by the exact excision
sequence
$$
A^0(D) \to \Pic(S) \to \Pic(U) \to 0
$$
of Picard groups (really Weil divisor class groups), cf.\  \cite[Cor.~21.6.10]{EGA4},
\cite[1~Prop.~1.8]{fulton:intersection_theory}.

By Theorem~\ref{prop:rank_4_sim}, we know that $\vp|_D :
\CliffB_0(q_1)|_E \isom \CliffB_0(q_2)|_E$ is induced by some
similarity transformation $\psi^D : q_{1}|_D \similar q_{2}|_D$.
Thus $\psi^K = (\psi^D|_K)\inv \circ \psi^U|_K \in \GOrth(q_1|_K)$.
Since $\CliffAlg_0(\psi^U|_K) = \CliffAlg_0(\psi^D|_K) = \vp|_K$, we
have that $\psi^K \in \GOrth(q_1|_K)$ is a homothety, multiplication
by $\lambda \in K\mult$.
As in \S\ref{sec:gluing}, define a line bundle $\PP$ on $S$ by the
gluing datum $(\OO_U,\OO_D,\lambda\inv : \OO_U|_K \isom \OO_D|_K)$.
Then $\PP$ comes equipped with isomorphisms $\rho^U : \OO_U \isom
\PP|_U$ and $\rho^D : \OO_D \isom \PP|_D$ with $(\rho^D|_K)\inv \circ
\rho^U|_K = \lambda\inv$.  Then we have similarities $\psi^U\tensor
\rho^U : q_1|_U \similar q_2|_D \tensor \PP|_U$ and $\psi^D\tensor
\rho^D : q_1|_D \similar q_2|_D \tensor \PP|_D$ such that
$$
(\psi^D\tensor \rho^D)|_K\inv \circ (\psi^U\tensor \rho^U)|_K =
(\psi^D|_K\inv \psi^U|_K) (\rho^D|_K\inv \rho^U|_K) =
\psi^K \lambda\inv  = \id
$$ 
in $\GOrth(q_1|_K)$.  Hence, as in \S\ref{sec:gluing}, $\psi^U\tensor
\rho^U$ and $\psi^D\tensor \rho^D$ glue to a similarity
$(\EE_1,q_1,\LL_1) \similar (\EE_2\tensor
\PP,q_2\tensor\qf{1},\LL_2\tensor\PP^{\tensor 2})$.  
Thus $(\EE_1,q_1,\LL_1)$ and $(\EE_2,q_2,\LL_2)$ define the same
element of $\Quadrics_2(T/S)$.
\end{proof}

\section{Failure of the local-global principle for isotropy of quadratic
forms}
\label{sec:failure}

In this section, we mention one application of the theory of quadratic
forms with simple degeneration over surfaces.  Let $S$ be a regular
proper integral scheme of dimension $d$ over an algebraically closed
field $k$ of characteristic $\neq 2$.  For a point $x$ of $X$, denote
by $K_x$ the fraction field of the completion $\widehat{\OO}_{S,x}$ of
$\OO_{S,x}$ at its maximal ideal.

\begin{lemma}
Let $S$ be a regular integral scheme of dimension $d$ over an
algebraically closed field $k$ of characteristic $\neq 2$ and let $D
\subset S$ be a divisor.  Fix $i > 0$.  If $(\EE,q,\LL)$ is a
quadratic form of rank $> 2^{d-i}+1$ over $S$ with simple degeneration
along $D$ then $q$ is isotropic over $K_x$ for all points $x$ of $S$ of
codimension $\geq i$.
\end{lemma}
\begin{proof}
The residue field $\kappa(x)$ of $K_x$ has transcendence degree $\leq
d-i$ over $k$ and is hence a $C_{d-i}$-field.  By hypothesis, $q$ has,
over $K_x$, a subform $q_1$ of rank $> 2^{d-i}$ that is regular over
$\widehat{\OO}_{S,x}$.  Hence $q_1$ is isotropic over $\kappa(x)$, thus $q$
is isotropic over the complete field $K_x$.
\end{proof}

As usual, denote by $K=k(S)$ the function field.  We say that a
quadratic form $q$ over $K$ is \linedef{locally isotropic} if $q$ is
isotropic over $K_x$ for all points $x$ of codimension one. 

\begin{cor}
\label{cor:local_isotropic}
Let $S$ be a proper regular integral surface over an algebraically
closed field $k$ of characteristic $\neq 2$ and let $D \subset S$ be a
regular divisor.  If $(\EE,q,\LL)$ is a quadratic form of rank $\geq
4$ over $S$ with simple degeneration along $D$ then $q$ over $K$ is
locally isotropic.
\end{cor}

For a different proof of this corollary, see
\cite[\S3]{parimala_ojanguren:quadratic_forms_complete_local_rings}.
However, quadratic forms with simple degeneration are mostly
anisotropic.

\begin{theorem}
\label{thm:anisotropic}
Let $S$ be a proper regular integral surface with 2 invertible and
$\Brtwo(S)$ trivial.  Let $T \to S$ be a finite flat morphism of degree 2 and
$D \subset S$ a smooth divisor.  Each nontrivial class in $\Brtwo(T)$
gives rise to an anisotropic quadratic form over $K$, unique up to similarity.
\end{theorem}
\begin{proof}
Let $L = k(T)$ and $K=k(S)$.  Let $\beta \in \Brtwo(T)$ be nontrivial.
Then by \cite{artin:Brauer-Severi}, $\beta_L \in \Brtwo(L)$ has index
2 and by purity for division algebras over regular surfaces
(Theorem~\ref{thm:purity}), there exists an Azumaya quaternion algebra
$\BB$ over $T$ whose Brauer class is $\beta$.  Since
$\Norm_{L/K}(\beta|_L)$ is unramified on $S$, by
Lemma~\ref{lem:cores_extends}, it extends to an element of
$\Brtwo(S)$, which is assumed to be trivial.  By the classical theory
of ${}^2\Dynkin{A}_1 = \Dynkin{D}_2$ over $L/K$ (cf.\
Theorem~\ref{thm:etale}), the quaternion algebra $\BB|_L$ corresponds
to a unique similarity class of quadratic form $q^K$ of rank 4 on $K$.
By Theorem~\ref{thm:main}, $\BB$ corresponds to a unique projective similarity
class of quadratic form $(\EE,q,\LL)$ of rank 4 with simple
degeneration along $D$ and such that $q|_K = q^K$.  

A classical result in the theory of quadratic forms of rank 4 is that
$q^K$ is isotropic over $K$ if and only if $\CliffB_0(q^K)$ splits
over $L$ (here $L/K$ is the discriminant extension of $q^K$), see
\cite[Thm.~6.3]{knus_parimala_sridharan:rank_4},
\cite[2~Thm.~14.1,~Lemma~14.2]{scharlau:book}, or \cite[II Prop.\
5.3]{baeza:semilocal_rings} in characteristic $2$.  Hence $q^K$ is
anisotropic since $\CliffB_0(q^K) = \BB_L$ has nontrivial Brauer class
$\beta$ by assumption.
\end{proof}

We make Theorem~\ref{thm:anisotropic} explicit as follows.  Write
$L=K(\sqrt{d})$.  Let $\BB$ be an Azumaya quaternion algebra over $T$
with $\BB_L$ given by the quaternion symbol $(a,b)$ over $L$.  Since
$\Norm_{L/K}(\BB_L)$ is trivial, the restriction-corestriction
sequence shows that $\BB|_L$ is the restriction of a class from
$\Brtwo(K)$, so we can choose $a,b \in K\mult$.  The corresponding
quadratic form over $K$ given by Theorem~\ref{thm:main} is then given,
up to similarity, by $\qf{1,a,b,abd}$, since it has discriminant $d$
and even Clifford invariant $(a,b)$ over $L$, see
\cite{knus_parimala_sridharan:rank_4}.

\begin{example}
Let $T \to \P^2$ be a branched double cover over a smooth sextic curve
over an algebraically closed field of characteristic zero.  Then $T$
is a smooth projective K3 surface of degree 2.  The Picard rank $\rho$
of $T$ can take values $1 \leq \rho \leq 20$, with the ``generic''
such surface having $\rho=1$.  In particular, $\Brtwo(T) \approx
(\Z/2\Z)^{22-\rho} \neq 0$, so that $T$ gives rise to $2^{22-\rho} -
1$ similarity classes of locally isotropic yet anisotropic quadratic
forms of rank 4 over $K=k(\P^2)$.  Explicit examples of such a
quadratic form are given in \cite{hassett_varilly_varilly} and
\cite{ABBV:pfaffian}.
\end{example}

\section{The Torelli theorem for cubic fourfolds containing a plane}
\label{sec:hodge}

Let $Y$ be a \linedef{cubic fourfold}, i.e., a smooth projective cubic
hypersurface of $\P^5 = \P(V)$ over $\C$.
Let $W \subset V$ be a subspace of rank three, $P=\P(W) \subset \P(V)$
the associated plane, and $P'=\P(V/W)$.  If $Y$ contains $P$, let
$\wt{Y}$ be the blow-up of $Y$ along $P$ and $\pi : \wt{Y} \to P'$ 
the projection from $P$.
The blow-up of $\P^5$ along $P$ is isomorphic to the total space of
the projective bundle $p : \P(\EE) \to P'$, where $\EE =
W\tensor\OO_{P'} \oplus \OO_{P'}(-1)$, and in which
$\pi : \wt{Y} \to P'$ embeds as a quadric surface bundle.  The
degeneration divisor of $\pi$ is a sextic curve $D \subset P'$.  It is
known that $D$ is smooth and $\pi$ has simple degeneration along $D$
if and only if $Y$ does not contain any other plane meeting $P$, cf.~\cite[\S1,~Lemme~2]{voisin:cubic_fourfolds}.  In this case, the
discriminant cover $T \to P'$ is a K3 surface of degree 2.
All K3 surfaces considered will be smooth and projective.

We choose an identification $P' = \P^2$ and suppose, for the rest of
this section, that $\pi : \wt{Y} \to P' = \P^2$ has simple
degeneration.  If $Y$ contains another plane $R$ disjoint from $P$,
then $R \subset \wt{Y}$ is the image of a section of $\pi$, hence
$\CliffB_0(\pi)$ has trivial Brauer class over $T$ by a
classical result concerning quadratic forms of rank 4, cf.\ proof of Theorem~\ref{thm:anisotropic}.  Thus if
$\CliffB_0(\pi)$ has nontrivial Brauer class $\beta \in \Brtwo(T)$,
then $P$ is the unique plane contained in $Y$.

Given a scheme $T$ with 2 invertible and an Azumaya quaternion algebra
$\BB$ on $T$, there is a canonically lift $[\BB] \in \Het^2(T,\muu_2)$
of the Brauer class of $\BB$, defined in
\cite{parimala_srinivas:brauer_group_involution} by taking into
account the standard symplectic involution on $\BB$.  Denote by $c_1 :
\Pic(T) \to \Het^2(T,\muu_2)$ the mod 2 cycle class map arising from
the Kummer sequence.

\begin{defin}
Let $T$ be a K3 surface of degree 2 over $k$ together with a
\linedef{polarization} $\FF$, i.e., an ample line bundle of
self-intersection 2.  For $\beta \in \Het^2(T,\muu_2)/\langle
c_1(\FF)\rangle$, we say that a cubic fourfold $Y$
\linedef{represents} $\beta$ if $Y$ contains a plane whose associated
quadric bundle $\pi : \wt{Y} \to \P^2$ has simple degeneration and
discriminant cover $f : T \to \P^2$ satisfying $f\pullback
\OO_{\P^2}(1) \isom \FF$ and $[\CliffB_0(\pi)]=\beta$.
\end{defin}

\begin{remark}
For a K3 surface $T$ of degree 2 with a polarization $\FF$, not every
class in $\Het^2(T,\muu_2)/\langle c_1(\FF) \rangle$ is represented by
a cubic fourfold, though one can characterize such classes.  Consider
the cup product mapping $\Het^2(T,\muu_2) \times \Het^2(T,\muu_2) \to
\Het^4(T,\muu_2^{\tensor 2}) \isom \Z/2\Z$.  Define
$$
B(T,\FF) = \{ \ol{x} \in \Het^2(T,\muu_2) / \langle c_1(\FF)\rangle \; \bigl| \; x \cup c_1(\FF) \neq 0\}.
$$
Note that the natural map $B(T,\FF) \to \Brtwo(T)$ is injective if and
only if $\Pic(T)$ is generated by $\FF$.  A consequence of the global
description of the period domain for cubic fourfolds containing a
plane is that for a K3 surface $T$ of degree 2 with polarization
$\FF$, the subset of $\Het^2(T,\muu_2)/\langle c_1(\FF) \rangle$ represented by a cubic
fourfolds containing a plane coincides with $B(T,f) \cup \{0\}$, cf.~\cite[\S9.7]{van_geemen:Brauer_K3} and
\cite[Prop.~2.1]{hassett_varilly_varilly}.
\end{remark}

We can now state the main result of this section.  Using Theorem~\ref{thm:main} and
results on twisted sheaves described below, we provide an algebraic
proof of the following result, which is due to Voisin
\cite{voisin:cubic_fourfolds} (cf.\
\cite[\S9.7]{van_geemen:Brauer_K3} and
\cite[Prop.~2.1]{hassett_varilly_varilly}).

\begin{theorem}
\label{thm:Torelli}
Let $T$ be a generic K3 surface of degree 2 with a polarization $\FF$.
Then each element of
$B(T,\FF)$
is represented by 
a single cubic fourfold containing a plane up to isomorphism.
\end{theorem}

We now explain the interest in this statement.  The global Torelli
theorem for cubic fourfolds states that a cubic fourfold $Y$ is
determined up to isomorphic by the polarized Hodge structure on
$H^4(Y,\Z)$.  Here \linedef{polarization} means a class $h^2 \in
H^4(Y,\Z)$ of self-intersection 3.  Voisin's approach
\cite{voisin:cubic_fourfolds} is to deal first with cubic fourfolds
containing a plane, then apply a deformation argument to handle the
general case.
For cubic fourfolds containing a plane, we can give an alternate
formulation, assuming the global Torelli theorem for K3 surfaces of
degree 2, which is a celebrated result of Piatetski-Shapiro and
Shafarevich \cite{p-s_shafarevich:torelli_K3}.

\begin{prop}
Assume the global Torelli theorem holds for a K3 surface $T$ of degree
2.  If the statement of Theorem~\ref{thm:Torelli} holds for $T$ then
the global Torelli theorem holds for all cubic fourfolds containing a
plane with $T$ as associated discriminant cover.
\end{prop}
\begin{proof}
Let $Y$ be a cubic fourfold containing a plane $P$ with 
discriminant cover $f : T \to \P^2$ and even Clifford algebra
$\CliffB_0$.
Consider the cycle class of $P$ in $H^4(Y,\Z)$.
Then $\FF=f\pullback\OO_{\P^2}(1)$ is a polarization on $T$, which
together with $[\CliffB_0] \in \Het^2(T,\muu_2)$, determines the
sublattice $\langle h^2,P \rangle^{\perp} \subset H^4(Y,\Z)$.
The key lattice-theoretic result we
use is \cite[\S1,~Prop.~3]{voisin:cubic_fourfolds}, which can be
stated as follows:\ the polarized Hodge structure $H^2(T,\Z)$ and the
class $[\CliffB_0] \in \Het^2(T,\muu_2)$ determines the Hodge
structure of $Y$; conversely, the polarized Hodge structure
$H^4(Y,\Z)$ and the sublattice $\langle h^2,P \rangle$ determines the
primitive Hodge structure of $T$, hence $T$ itself by the global
Torelli theorem for K3 surfaces of degree 2.

Now let $Y$ and $Y'$ be cubic fourfolds containing a plane $P$ with
associated discriminant covers $T$ and $T'$ and even Clifford algebras
$\CliffB_0$ and $\CliffB_0'$.  Assume that $\Psi : H^4(Y,\Z) \isom
H^4(Y,\Z)$ is an isomorphism of Hodge structures preserving the
polarization $h^2$.  By \cite[Prop.~3.2.4]{hassett:special}, we can
assume (by composing $\Psi$ with a Hodge automorphism fixing $h^2$)
that $\Psi$ preserves the sublattice $\langle h^2,P \rangle$.  By
\cite[\S1,~Prop.~3]{voisin:cubic_fourfolds}, $\Psi$ induces an
isomorphism $T\isom T'$, with respect to which
$[\CliffB_0]=[\CliffB_0']=\beta \in \Het^2(T,\muu_2) \isom
\Het^2(T',\muu_2)$.  Hence if there is at most a single cubic fourfold
representing $\beta$ up to isomorphism then $Y\isom Y'$.
\end{proof}

The following lemma, whose proof we could not find in the literature,
holds for smooth cubic hypersurfaces $Y \subset \P_k^{2r+1}$
containing a linear subspace of dimension $r$ over any field $k$.
Since $\Aut(\P_k^{2r+1}) \isom \PGL_{2r+2}(k)$ acts transitively on
the set of linear subspaces in $\P_k^{2r+1}$ of dimension $r$, any two
cubic hypersurfaces containing linear subspaces of dimension $r$ have
isomorphic representatives containing a common such linear subspace.

\begin{lemma}
\label{lem:isomorphisms}
Let $Y_1$ and $Y_2$ be smooth cubic hypersurfaces in $\P_k^{2r+1}$
containing a linear space $P$ of dimension $r$.  The associated
quadric bundles $\pi_1 : \wt{Y}_1 \to \P_k^r$ and $\pi_2 : \wt{Y}_2 \to
\P_k^r$ are $\P_k^r$-isomorphic if and only if the there is a linear
isomorphism $Y_1 \isom Y_2$ fixing $P$.
\end{lemma}
\begin{proof}
Any linear isomorphism $Y_1 \isom Y_2$ fixing $P$ will induce an
isomorphism of blow-ups $\wt{Y}_1 \isom \wt{Y}_2$ commuting with the
projections from $P$.  Conversely, assume that $\wt{Y}_1$ and $\wt{Y}_2$ are
$\P_k^r$-isomorphic.  
Since $\PGL_{2r+2}(k)$ acts transitively on the
set of linear subspaces of dimension $r$, without loss of generality,
we can assume that $P = \{ x_0=\dotsm=x_r=0\}$ where
$(x_0:\dotsm:x_r:y_0:\dotsm:y_r)$ are homogeneous coordinates on
$\P_k^{2r+1}$.  For $l=1,2$, write $Y_l$ as
$$
\sum_{0 \leq m \leq n \leq r} a_{mn}^l y_m y_n +
\sum_{p=0}^r b_p^l y_p + c^l = 0
$$ 
for homogeneous linear forms $a_{mn}^l$, quadratic forms $b_p^l$, and
cubic forms $c^l$ in $k[x_0,\dotsc,x_r]$.  The blow-up of $\P_k^{2r+1}$
along $P$ is identified with the total space of the projective bundle
$\pi : \P(\EE)\to \P_k^r$, where $\EE =
\OO_{\P_k^r}^{r+1}\oplus\OO_{\P_k^r}(-1)$.  The homogeneous coordinates
$y_0,\dots,y_r$ correspond, in the blow-up, to a basis of global
sections of $\OO_{\P(\EE)}(1)$.  Let $z$ be a nonzero global section
of $\OO_{\P(\EE)}(1)\tensor\pi\pullback\OO_{\P_k^r}(-1)$.
Then $z$ is unique up to scaling, as we have
$$
\Gamma(\P(\EE),\OO_{\P(\EE)}(1)\tensor\pi\pullback\OO_{\P_k^r}(-1))
\isom
\Gamma(\P_k^r,\pi\pushforward\OO_{\P(\EE)}(1)\tensor\OO_{\P_k^r}(-1)) = 
\Gamma(\P_k^r,\EE\dual \tensor \OO_{\P_k^r}(-1)) = k
$$
by the projection formula.  Thus $(y_0:\dotsm:y_r:z)$ forms a relative
system of homogeneous coordinates on $\P(\EE)$ over $\P_k^r$.  Then
$\wt{Y}_l$ can be identified with the subscheme of $\P(\EE)$ defined
by the global section
$$
q_l(y_0,\dotsc,y_r,z)=\sum_{0 \leq m \leq n \leq r} a_{mn}^l y_m y_n +
\sum_{0 \leq p \leq r} b_p^l y_p z + c^lz^2 = 0
$$ 
of $\OO_{\P(\EE)}(2)\tensor\pi\pullback\OO_{\P_k^r}(1)$.  Under these
identifications, $\pi_l : \wt{Y}_l \to \P_k^r$ can be identified with
the restriction of $\pi$ to $\wt{Y}_l$, hence with the quadric bundle
associated to the line bundle-valued quadratic form $(\EE,
q_l,\OO_{\P_k^r}(1))$.
Since $Y_l$ and $P$ are smooth, so is $\wt{Y}_l$.  Thus $\pi_l :
\wt{Y}_l \to \P_k^r$ is flat, being a morphism from a Cohen--Macaulay
scheme to a regular scheme.  Thus by
Propositions~\ref{prop:proj_sim_quadric} and \ref{prop:exactness}, the
$\P_k^r$-isomorphism $\wt{Y}_1\isom\wt{Y}_2$ induces a projective
similarity $\psi$ between $q_1$ and $q_2$.  But as $\EE \tensor \NN
\isom \EE$ implies $\NN$ is trivial in $\Pic(\P_k^r)$, we have that
$\psi : q_1 \similar q_2$ is, in fact, a similarity.  In particular,
$\psi \in \GL(\EE)(\P_k^r)$, hence consists of a block matrix of the
form
$$
\begin{pmatrix}
H & v \\
0 & u
\end{pmatrix}
$$
where $H \in \GL(\OO_{\P_k^r}^{r+1})(\P_k^r) = \GL_{r+1}(k)$ and $u \in
\GL(\OO_{\P_k^r}(-1))(\P_k^r) = \Gm(\P_k^r) = k\mult$, while $v \in
\Hom_{\OO_{\P_k^r}}(\OO_{\P_k^r}(-1),\OO_{\P_k^r}^{r+1})$ $=
\Gamma(\P_k^r,\OO_{\P_k^r}(1))^{\oplus (r+1)}$ consists of a vector of
linear forms in $k[x_0,\dotsc,x_r]$.  Let $v = G \cdot
(x_0,\dotsc,x_r)^t$ for a matrix $G \in M_{r+1}(k)$.  Then writing $H
= (h_{ij})$ and $G = (g_{ij})$, we have that  
$\psi$ acts as 
$$
x_i \mapsto x_i, \qquad 
y_i \mapsto \sum_{0 \leq j \leq r} (h_{ij}y_j + g_{ij}x_jz), \qquad 
z \mapsto uz 
$$
and satisfies $q_2(\psi(y_0),\dotsc,\psi(y_r),\psi(z))=\lambda
q_1(y_0,\dotsc,y_r,z)$ for some $\lambda \in k\mult$.  Considering the matrix $J \in M_{2r+2}(k)$ with
$(r+1)\times(r+1)$ blocks
$$
J=
\begin{pmatrix}
uI & 0 \\
G  & H
\end{pmatrix}
$$
as a linear automorphism of $\P_k^{2r+1}$, then $J$ acts on $(x_0:\dotsm:x_r:y_0:\dotsm:y_r)$ as  
$$
x_i \mapsto ux_i, \qquad 
y_i \mapsto \sum_{0 \leq j \leq r} (h_{ij}y_j + g_{ij}x_j),
$$
and hence satisfies $q_2(J(y_0),\dotsc,J(y_r),1)=u\lambda
q_1(y_0,\dotsc,y_r,1)$ due to the homogeneity properties of $x_i$ and
$z$.  Thus $J$ is a linear automorphism taking $Y_1$ to $Y_2$ and
fixes $P$.
\end{proof}

Let $T$ be a K3 surface.
We shall freely use the notions of $\beta$-twisted sheaves, $B$-fields
associated to $\beta$, the $\beta$-twisted Chern character, and
$\beta$-twisted Mukai vectors from \cite{huybrechts_stellari}.  For a
Brauer class $\beta \in \Brtwo(T)$ we choose the rational $B$-field
$\beta/2 \in H^2(T,\Q)$. The $\beta$-twisted Mukai vector of a
$\beta$-twisted sheaf $\VV$ is
$$
v^B(\VV) = \text{ch}^B(\VV) \sqrt{\text{Td}_T} = \bigl(\rk \VV,
c_1^B(\VV), \rk \VV + \frac{1}{2}c_1^B(\VV) - c_2^B(\VV) \bigr) 
\in H^*(T,\Q)
$$ 
where $H^*(T,\Q)=\bigoplus_{i=0}^2 H^{2i}(T,\Q)$.  As in
\cite{mukai:moduli_bundles_K3}, one introduces the
Mukai pairing 
$$
(v,w) = v_2 \cup w_2 - v_0\cup w_4 - v_4 \cup w_0 \in H^4(T,\Q)\isom \Q
$$ 
for Mukai vectors $v = (v_0,v_2,v_4)$ and $w = (w_0,w_2,w_4)$.

By \cite[Thm.~3.16]{yoshioka:twisted_sheaves}, the moduli space of
stable $\beta$-twisted sheaves $\VV$ with Mukai vector $v=v^B(\VV)$
satisfying $(v,v)=2n$ is isomorphic to the Hilbert scheme
$\text{Hilb}^{{n}+1}_T$.  In particular, when $(v,v)=-2$, this moduli
space consists of one point; we give a direct proof of this
fact inspired by \cite[Cor.~3.6]{mukai:moduli_bundles_K3}.

\begin{lemma}
\label{lem:Exts}
Let $T$ be a K3 surface and $\beta \in \Brtwo(T)$ with chosen
$B$-field. Let $v \in H^*(T,\Q)$ with $(v,v)=-2$.  If $\VV$ and $\VV'$
are stable $\beta$-twisted sheaves with $v^B(\VV)=v^B(\VV')=v$ then
$\VV \isom \VV'$.
\end{lemma}
\begin{proof}
Assume that $\beta$-twisted sheaves $\VV$ and $\VV'$ have the same
Mukai vector $v \in H^2(T,\Q)$.  Since $-2 = (v,v) =
\chi(\VV,\VV)=\chi(\VV,\VV')$, a Riemann--Roch calculation shows that
either $\Hom(\VV,\VV') \neq 0$ or $\Hom(\VV,\VV') \neq 0$.  Without
loss of generality, assume $\Hom(\VV,\VV') \neq 0$.  Then since $\VV$
is stable, a nonzero map $\VV \to \VV'$ must be injective.  Since $\VV'$ is
stable, the map is an isomorphism.
\end{proof}

\begin{lemma}
\label{lem:stable}
Let $T$ be a K3 surface of degree 2 and $\beta \in \Brtwo(T)$ with
chosen $B$-field.  Let $Y$ be a smooth cubic fourfold containing a
plane whose even Clifford algebra $\CliffB_0$ represents $\beta \in
\Brtwo(T)$.  If $\VV_0$ is a $\beta$-twisted sheaf associated to
$\CliffB_0$ then $(v^B(\VV_0),v^B(\VV_0))=-2$. Furthermore, if $T$ is
generic then $\VV_0$ is stable.
\end{lemma}
\begin{proof}
By the $\beta$-twisted Riemann--Roch theorem, we have
$$
- (v^B(\VV_0),v^B(\VV_0) ) = \chi(\VV_0,\VV_0) = \sum_{i=0}^2 \Ext^i_T(\VV_0,\VV_0).
$$
Then $v^B(\VV_0)=2$ results from the fact that $\VV_0$ is a
\linedef{spherical} object, i.e., $\Ext^i_T(\VV_0,\VV_0)=\C$ for
$i=0,2$ and $\Ext^1(\VV_0,\VV_0)=0$.  Indeed, as in \cite[Rem.~2.1]{macri_stellari},  we have
$\Ext_T^i(\VV_0,\VV_0)=H^i(\P^2,\CliffAlg_0)$, which can be calculated directly
using the fact that, as $\OO_{\P^2}$-algebras,  
\begin{equation}
\label{eq:CliffAlg}
\CliffB_0 \isom \OO_{\P^2} \oplus \OO_{\P^2} (-3) \oplus \OO_{\P^2} (-1)^3 \oplus \OO_{\P^2} (-2)^3
\end{equation} 
If $T$
is generic, stability follows from
\cite[Prop.~3.14]{mukai:moduli_bundles_K3}, see also \cite[Prop.~3.12]{yoshioka:twisted_sheaves}.
\end{proof}

\begin{lemma}
\label{lem:twisted_unique}
Let $T$ be a K3 surface of degree 2.  Let $Y$ and $Y'$ be
smooth cubic fourfolds containing a plane whose respective
even Clifford algebras $\CliffB_0$ and $\CliffB_0'$ represent the same
$\beta \in \Brtwo(T)$.  If 
$T$ is generic then $\CliffB_0 \isom \CliffB_0'$.
\end{lemma}
\begin{proof}
Let $\VV_0$ and $\VV_0'$ be $\beta$-twisted sheaves associated to
$\CliffB_0$ and $\CliffB_0$, respectively.  A consequence of
\cite[Lemma~3.1]{macri_stellari} and \eqref{eq:CliffAlg} is that
$v=v^B(\VV_0)=v^{B}(\VV_0'\tensor \NN)$ for some line bundle $\NN$ on
$T$.  Replacing $\VV_0'$ by $\VV_0'\tensor\NN\dual$, we can assume
that $v^B(\VV_0)=v^{B}(\VV_0')$.  By Lemma~\ref{lem:stable}, we have
$(v,v) = -2$ and that $\VV_0$ and $\VV_0'$ are stable.  Hence by
Lemma~\ref{lem:Exts},
we have $\VV_0\isom\VV_0'$ as $\beta$-twisted
sheaves, hence $\CliffB_0 \isom \EEnd(\VV_0) \isom \EEnd(\VV_0') \isom
\CliffB_0'$.
\end{proof}

\begin{proof}[Proof of Theorem~\ref{thm:Torelli}]
Suppose that $Y$ and $Y'$ are smooth cubic fourfolds containing a
plane whose associated even Clifford algebras $\CliffB_0$ and
$\CliffB_0'$ represent the same class $\beta \in B(T,\FF) \subset
\Het^2(T,\muu_2)/\langle c_1(\FF) \rangle \isom \Brtwo(T)$.  By
Lemma~\ref{lem:twisted_unique}, we have $\CliffB_0 \isom \CliffB_0'$.
By Theorem~\ref{thm:main}, the quadric surface bundles $\pi : \wt{Y} \to \P^2$ and
$\pi' : \wt{Y}' \to \P^2$ are $\P^2$-isomorphic.  Finally, by
Lemma~\ref{lem:isomorphisms}, we have $Y \isom Y'$.
\end{proof}

\providecommand{\bysame}{\leavevmode\hbox to3em{\hrulefill}\thinspace}

\end{document}